\documentclass[english,12pt]{article}
\usepackage[utf8]{inputenc}
\usepackage{amssymb,amsmath,amsfonts,amsthm,rotating,setspace,graphicx,float,booktabs,caption}
\usepackage{amssymb,amsmath,amsthm,rotating,setspace,verbatim}
\usepackage{enumerate}
\usepackage{multirow}
\usepackage{cancel}
\usepackage{lscape}
\usepackage{tikz,pifont}
\usepackage{esint}
\usepackage{mathrsfs}
\usepackage{geometry,url}
\usepackage{bbm}
\usepackage{paralist}
\usepackage[hidelinks]{hyperref}
\usepackage[title]{appendix}
\usepackage{latexsym}
\usepackage{epsfig}
\usepackage{color}

\usepackage{authblk,  empheq}

\newtheorem{thm}{Theorem}[section]
\newtheorem{lemma}[thm]{Lemma}
\newtheorem{prop}[thm]{Proposition}
\newtheorem{cor}[thm]{Corollary}

\theoremstyle{definition}
\theoremstyle{definition}

\numberwithin{equation}{section}

\def\K1{K_{\alpha, \beta}}

\def\B2x{B_{2R}(x)}

\def\Ix{I_{\alpha,\beta,\gamma}(\rho,x)}
\def\Jx{J_{\alpha,\beta,\theta}(\rho,x)}
\def\ds{\displaystyle}

\def\R{\mathbbm{R}}

\def\B1{B_1}
\newcommand{\abs}[1]{\lvert#1\rvert}

\title{Isolated singularities for elliptic equations with convolution terms in a punctured ball}
\author{Marius Ghergu\footnote{School of Mathematics and Statistics,
        University College Dublin, Belfield, Dublin 4, Ireland;
        {\tt marius.ghergu@ucd.ie}; ORCID: 0000-0001-9104-5295} \;\footnote{Institute of Mathematics Simion Stoilow of the Romanian Academy, 21 Calea Grivitei St., 010702 Bucharest, Romania}
        $\;\;$            and    $\;$
    Zhe Yu\footnote{School of Mathematics and Statistics,
        University College Dublin, Belfield, Dublin 4, Ireland;
        {\tt zhe.yu@ucdconnect.ie }; ORCID: 0000-0002-2261-4981} }

\date{}
\begin{document}

\maketitle

\begin{abstract}
\vspace{0.3cm} 
\noindent The purpose of this article is two-fold. First, we investigate the inequality
$$
-\Delta u+V(x) u\geq f\quad\mbox{ in } B_1\setminus\{0\}\subset \mathbbm{R}^N , N \geq 2,
$$
where $f\in L^1_{loc}(B_1)$. If $V\geq 0$ is radially symmetric, we provide optimal conditions for which any solution $0\leq u\in \mathcal{C}^2(B_1\setminus\{0\})$ of the above inequality satisfies $u, \Delta u, V(x)u\in L^1_{loc}(B_1)$. This  extends a result of H. Brezis and P.-L. Lions (1982), originally established for constant potentials $V$. Second, we investigate the equation 
$$\displaystyle -\Delta u + \lambda V(x) u = (K_{\alpha, \beta} * u^p) u^q \quad\text{in } B_1 \setminus \{0\},$$
where $0\leq V\in \mathcal{C}^{0, \nu}( \overline B_1\setminus\{0\})$, $0<\nu<1$, $\lambda, p, q>0$ and 
$$K_{\alpha, \beta}(x) = |x|^{-\alpha}\log^{\beta}\frac{2e}{|x|}, \quad\text{where } 0 \leq \alpha < N, \beta \in \mathbbm{R}.$$
For $N \geq 3$, we establish sharp conditions on the exponents $\alpha, \beta, p, q$ under which singular solutions exist and exhibit the asymptotic behavior $u(x) \simeq |x|^{2-N}$ near the origin. For $N = 2$, we provide a classification of the existence and boundedness of solutions based on the local behavior of the potential $V(x)$ near the origin.

\noindent 
\end{abstract}

\noindent{\bf Keywords:} Semi-linear elliptic equations; 
Singular solutions; Convolution term.
\medskip

\noindent{\bf 2020 AMS MSC:} 35J61; 35A23; 35A21; 35B33; 35A08; 35J75

\newpage
\section{Introduction}

The study of isolated singularities in semi-linear elliptic equations has been a central focus in mathematical analysis, particularly for understanding the behavior of solutions near the singularity point. As it is well known, the fundamental solutions of the Laplacian operator play a crucial role in many works. A celebrated result due to Brezis and Lions \cite{BL81} in the early 1980s states that\footnote{The solutions were considered in the class of distributions in the original article \cite{BL81}.}:

\medskip
\noindent{\bf Theorem A.} {\it If $0\leq u\in \mathcal{C}^2(B_1\setminus\{0\})$ is a solution of 
\begin{equation}\label{bl}
-\Delta u+au\geq f \quad\mbox{ in } B_1\setminus\{0\}
\end{equation}
with $f\in L^1_{loc}(B_1)$, then $u, \Delta u\in L^1_{loc}(B_1)$ and there exists $h \in L^1_{loc}(B_1)$ and a non-negative number $c\geq 0$ such that }
\begin{equation}\label{h}
-\Delta u=h+c\delta_0 \quad\mbox{ in } \mathscr{D}'(B_1).
\end{equation}
Lions \cite{L80} extended the above result by considering the following problem:
\begin{equation}\label{lions}
\left\{\begin{aligned}
   -\Delta u &= u^p &\text{in}\; &B_1\setminus \{0\}, \\
   u &= 0 &\text{on}\; &\partial B_1. 
\end{aligned}\right.
\end{equation}
It is shown in \cite{L80} that any solution $u$ of the above problem has either a removable singularity at the origin or it satisfies:
$$-\Delta u = u^p + \alpha\delta_0 \quad\text{in } B_1.$$ 
Moreover, if $p \geq \frac{N}{N-2}$, then $\alpha = 0$ and if $1 < p < \frac{N}{N-2}$, then there exists $\alpha_* > 0$ such that \eqref{lions} has: 
\begin{itemize}
    \item two classical solutions for $0 < \alpha < \alpha_*$,
    \item one classical solution for $\alpha = \alpha_*$,
    \item no solution for $\alpha > \alpha_*$. 
\end{itemize}
A more recent research by Chen and Zhou \cite{CZ18} extended the study of \eqref{lions} to the case of semi-linear elliptic equations involving Hardy operator:
\begin{equation}\label{CZ}
\left\{\begin{aligned}
   -\Delta u + \mu |x|^{-2}u&= u^p &\text{in}\; &B_1\setminus \{0\}, \\
   u &= 0 &\text{on}\; &\partial B_1. 
\end{aligned}\right.
\end{equation}
Similar to  Lions' work, Chen and Zhou classified the solutions of \eqref{CZ} based on new critical Hardy exponents (see also \cite{CQZ21, D02} for further extensions). \medskip\\
The purpose of the present work is two-fold. First, we want to extend the result in Theorem A to a more general case. Second, we are interested in the study of a non-local equation governed by a non-linear convolution term. 
This is motivated by the recent studies carried out in \cite{BP01,CZ16,CZ17,CZ23, CRSW25, FG20,G23,GKS21, GMM21, GT16, GY23, HR21, HRW22, JYZ25, LMSS25, LZX24, MT25, M25, PRS24,PN25, SR24, WS25,YY25}.\\
Before we proceed to the statements of our main results, let us make precise the notations used in this article:
\begin{itemize}
    \item Given a continuous function $u:B_1\setminus\{0\}\to \R$, we denote by $\overline u(r)$ its spherical average over the sphere $\partial B_r$, $0<r<1$, that is,
    $$\overline u(r)=\frac{1}{|\partial B_r|}\int_{\partial B_r} u(x) dS(x),$$
    where $S$ denotes the spherical measure in $\R^{N-1}$. 
    \item For two positive functions $f, g:B_1\setminus\{0\}\to (0, \infty$), we use the symbol $f(x)\simeq g(x)$ to denote that the quotient $f/g$ is bounded between two positive constants in $B_1\setminus\{0\}$.
    \item $\delta_0$ denotes the Dirac delta measure in $\R^N$. 
    \item By $C, C_1, C_2, c, ... $ we denote generic positive constants whose value may change on every occurrence.
\end{itemize}
We are now ready to state our main results. In relation to Theorem A, we have the following extension:
\begin{thm}\label{thm1}
    Let $0\leq u\in \mathcal{C}^2(B_1\setminus\{0\})$ be a solution of  
\begin{equation}\label{h1}
    -\Delta u+V(|x|) u\geq f \quad\mbox{ in }\; B_1\setminus\{0\},
\end{equation}
where $f\in \mathcal{C}(B_1\setminus\{0\})\cap L^1_{loc}(B_1)$ and $V: (0, 1] \mapsto (0, \infty)$ is continuous and satisfies
\begin{equation}\label{dini}
\left\{\begin{aligned}
    &\int_0^1 s\log\frac{1}{s}V(s) ds < \infty\quad&\text{if }N = 2,\\
    &\int_0^1 sV(s) ds < \infty \quad&\text{if }N \geq 3.
\end{aligned}\right.
\end{equation}
Then, 
\begin{equation}\label{cond}
u,\; \Delta u,\;  V(|x|)u \in L^1_{loc}(\B1),
\end{equation}
and there exists $h\in L^1_{loc}(B_1)$ such that \eqref{h} holds.
\end{thm}
\noindent
For $V\equiv a>0$, Theorem \ref{thm1} implies Theorem A above.
Note that $f \in L^1_{loc}(B_1)$ is a necessary condition for the existence of solutions to \eqref{h1} (see Lemma \ref{lm_3}). We should also note that condition \eqref{dini} is optimal at least when $f\equiv 0$. Indeed, let us denote
\begin{equation}\label{E}
E(r)=
\begin{cases}
\ds \frac{1}{2\pi} \log\frac{e}{r} &\quad\mbox{if }N=2,\\[0.1in]
\ds \frac{1}{(N-2)\sigma_N} r^{2-N} &\quad\mbox{if }N\geq 3,
\end{cases}
\end{equation}
where $\sigma_N$ denotes the surface area of the unit ball in $\R^N$, $N\geq 3$. Then $E(|x|)$ is the fundamental solution of the Laplace operator in $\R^N$, $N\geq 2$, in the sense that 
$$
-\Delta E=\delta_0\quad\mbox{ in } \mathscr{D}'(\R^N).
$$
Also, $u(x)=E(|x|)$ satisfies \eqref{h1} with $f\equiv 0$ and $V(|x|)E(|x|)\in L^1_{loc}(B_1)$ yields \eqref{dini}. We thus have the following result:

\begin{cor}\label{cor0}
Let $V: (0, 1] \mapsto (0, \infty)$ be a continuous function. Then, 
any solution $0\leq u\in \mathcal{C}^2(B_1\setminus\{0\})$ of  $-\Delta u+V(|x|) u\geq 0$ in  $B_1\setminus\{0\}$ satisfies 
$$
u,\; \Delta u,\;  V(|x|)u  \in L^1_{loc}(\B1),
$$
if and only if \eqref{dini} holds.
\end{cor}
\noindent To illustrate the usefulness of Theorem \ref{thm1}, let us consider the inequality 
\begin{equation}\label{usef}
-\Delta u+\frac{\log^\tau \frac{e}{|x|} }{|x|^\gamma} u \geq 0 \quad\mbox{ in }\; B_1\setminus\{0\}.    
\end{equation}
Using Theorem \ref{thm1} we obtain:
\begin{cor}\label{cor1}
Let $\gamma, \tau\in \R$. Then, any solution $0\leq u\in \mathcal{C}^2(B_1\setminus\{0\})$  
 of \eqref{usef}
satisfies 
$$
u,\; \Delta u,\;   \frac{\log^\tau \frac{e}{|x|} }{|x|^\gamma} u \in L^1_{loc}(\B1),
$$
if and only if one of the following conditions holds: 
    \begin{enumerate}[(i)]
        \item $N \geq 2$ and $\gamma < 2$; 
        \item $N = 2$, $\gamma = 2$ and $\tau < -2$;
        \item $N \geq 3$, $\gamma = 2$ and $\tau < -1$.
    \end{enumerate}
    \end{cor}
\noindent
The second goal of this article is concerned with the study of non-negative solutions to the following elliptic equation with convolution term:
\begin{equation}\label{eq0}
 -\Delta u +  \lambda V(x) u = (\K1  * u^p)u^q\quad \text{in } B_1\setminus\{0\}, 
\end{equation}
where $0< V\in \mathcal{C}^{0, \nu}( \overline B_1\setminus\{0\})$, $0<\nu<1$, is not necessarily radial, $\lambda, p, q>0$ and
\begin{equation}\label{ptt}
\K1(x) = \abs{x}^{-\alpha}\log^{\beta}\frac{2e}{\abs{x}}, \quad\text{where } 0 \leq \alpha < N,\, \beta \in \R.
\end{equation} 
We note that $\K1(x) \in L^1(B_1)$. The quantity $\K1 * u^p$ represents the convolution operation in $B_1$ given by
\[
(\K1 \ast u^p)(x)=\int_{B_1} \frac{\log^{\beta}\frac{2e}{\abs{x - y}}}{\abs{x-y}^{\alpha}} u^p(y) dy, \quad \text{for all } x \in B_1\setminus\{0\}.
\]
We say that a function $u \in \mathcal{C}^2(B_1 \setminus \{0\})$ is a solution to \eqref{eq0} if
\begin{equation}\label{KKK3}
u \geq 0\;, \;\; (\K1 * u^p)(x)<\infty \quad\mbox{ for all }x\in B_1\setminus\{0\},  
\end{equation}
and $u$ satisfies \eqref{eq0} pointwise in $B_1 \setminus \{0\}$. It is easy to see  that \eqref{KKK3} implies $u^p\in L^1(B_1)$. \\
If $\beta=0$, then $K_{\alpha, 0}$ becomes the standard Riesz potential of order $N-\alpha$ and the inequality $-\Delta u+\lambda V(x) u\geq (K_{\alpha, 0}*u^p)u^q$ was studied in the exterior domains of $\R^N$ in \cite{GKS20, GKS21, MV13}. Equation \eqref{eq0} with $\beta=0$ is often linked with the so-called Choquard equation arising in quantum physics models (see \cite{MV17} for an account on this topic). When $q=p-1$, the equation \eqref{eq0} has a variational structure and its energy functional is given by 
$$
I(u)=\frac{1}{2}\int_{B_1}\left(|\nabla u|^2+\lambda V(x)u^2\right)-\frac{1}{2p}\int_{B_1}\int_{B_1} K_{\alpha, \beta}(x-y) u^p(x)u^p(y) dxdy.
$$
There are many recent studies of the above and related functional for $\beta=0$; we will not proceed in this direction and will investigate \eqref{eq0} for general and independent exponents $p, q>0$.\medskip\\
Our first result concerning the non-negative solutions of \eqref{eq0} reads as follows:
\begin{thm}\label{thm2}
Suppose $0\leq u\in \mathcal{C}^2(B_1\setminus\{0\})$ is a solution of \eqref{eq0} such that $Vu\in L^1_{loc}(B_1)$. Then $(\K1 * u^p)u^q\in L^1_{loc}(B_1)$ and there exists $m\geq 0$ such that $u$ satisfies
\begin{equation}\label{eq1d}
-\Delta u +  \lambda V(x) u = (\K1  * u^p)u^q+m\delta_0 \quad \text{in } \mathscr{D}'(\B1).
\end{equation}
\end{thm}
\noindent Next, we will be concerned with solutions $0\leq u\in \mathcal{C}^2(B_1\setminus\{0\})$ of \eqref{eq0} that fulfill $u(x)\simeq E(|x|)$.
We first discuss the case $N=2$.
\begin{thm}\label{thm3} Assume $N=2$. 
\begin{enumerate}[(i)]
    \item If $V(x)\log\frac{2e}{|x|}\in L^1_{loc}(B_1)$ and $u\geq 0$ is a solution of \eqref{eq0} such that 
        \begin{equation}\label{eqN2}
        u(x)=o\left(\log\frac{2e}{|x|}\right)\quad\mbox{ as }|x|\to 0,
        \end{equation}
        then $u$ has a $\mathcal{C}^2$-extension at the origin. 
        \item Let $p>0$, $q>1$ and assume that for all $\varepsilon>0$ we have 
        \begin{equation}\label{eqN21}
        V(x)=O\left(\log^\varepsilon \frac{2e}{|x|}\right)\quad\mbox{ as } |x|\to 0.
        \end{equation}
        Then,  there exists $\lambda^*>0$ such that for all $0<\lambda<\lambda^*$, \eqref{eq0} has a positive solution $u(x)\simeq \log\frac{2e}{|x|}$.
    \end{enumerate}
    \end{thm}
\noindent Typical examples of potentials $V(x)$ that fulfill \eqref{eqN21} are
$$V(x)=\log^{-\theta}\frac{2e}{|x|}\, , \, \theta>0\, ,\qquad V(x)=a\geq 0\, ,\qquad V(x)=\log\log\frac{2e}{|x|}.$$
In the case $N\geq 3$ our result concerning \eqref{eq0} reads as follows. 
\begin{thm}\label{thm4} Assume $N\geq 3$ and
\begin{equation}\label{eqV}
V(x)=
\begin{cases}
o\left(|x|^{-(q-1)(N-2)}\right) & \mbox{ if }q>1,\\[0.2cm]
O(1) & \mbox{ if }0<q\leq 1,
\end{cases}\quad\mbox{ as }|x|\to 0.
\end{equation}
Then, there exists $\lambda>0$ and a solution $u(x)\simeq |x|^{2-N}$ of \eqref{eq0} if and only if one of the conditions below hold: 
\begin{enumerate}
   \item[\it (i)] either 
       \begin{equation}\label{eqN3}
       \max\{p,q\}<\frac{N}{N-2} \, , \quad  \quad p+q<\frac{2N-\alpha}{N-2} \quad\mbox{ and } \quad \beta\in \R,
       \end{equation}   
   \item[\it (ii)] or
       \begin{equation}\label{eqN4}
       \max\{p,q\}<\frac{N}{N-2}  \, ,\qquad p+q=\frac{2N-\alpha}{N-2} \qquad\mbox{ and } \quad \beta < -1.
       \end{equation}           
\end{enumerate}
\end{thm}
\noindent We point out that if $0\leq u\in \mathcal{C}^2(B_1\setminus\{0\})$ is a solution of \eqref{eq0} such that $u(x)\simeq |x|^{2-N}$, then \eqref{eqN3}, \eqref{eqN4} and \eqref{eqV} imply $Vu\in L^1_{loc}(B_1)$. By Theorem \ref{thm2}, $u$ satisfies \eqref{eq1d}.

\noindent In relation to our Theorem \ref{thm3} above, it is worth mentioning that the problem 
\begin{equation}\label{hc1}
\left\{\begin{aligned}
&-\Delta u+u=(|x|^{-\alpha}*u^p)u^q\\[0.2cm]
&\lim\limits_{|x|\to \infty} u(x)=0
\end{aligned}\right.\quad\mbox{ in }\R^N\setminus\{0\} , N\geq 3,
\end{equation}
was studied in \cite{CZ16} (for $p>0$ and $q\geq 1$) and  in \cite{CZ17} (for $p>0$ and $0<q<1$). It is obtained in \cite[Theorem 1.1]{CZ17} that if
\begin{equation}\label{eqhc}
\frac{\alpha}{N}p+q<1 \quad\mbox{ and }\quad p+q<\frac{2N-\alpha-2}{N-2},
\end{equation}
then \eqref{hc1} has no solutions. 
In contrast to this fact, Theorem \ref{thm4} above allows us to deduce the existence of a solution to the local problem \eqref{eq0} for all ranges of exponents $p, q>0$. Precisely, we have:

\begin{cor}\label{cor2}
Assume $N\geq 3$ and $p$, $q$ satisfy either \eqref{eqN3} or \eqref{eqN4}. Then, for all $\lambda>0$ if $p+q\neq 1$ (resp. for all $\lambda>0$ large, if $p+q=1$) the problem 
\begin{equation}\label{lala}
\begin{cases}
-\Delta u +  \lambda  u = (\K1  * u^p)u^q  \quad \text{in } B_1\setminus\{0\}  ,\\[0.2cm]
\displaystyle \lim_{|x|\to 0} \frac{u(x)}{|x|^{2-N}}>0,
\end{cases}
\end{equation}
has a positive solution $u\in \mathcal{C}^2(B_1\setminus\{0\})$.
\end{cor}
\noindent The approach in \cite{CZ16, CZ17} relies on an iteration process combined with the properties of the fundamental solution of $-\Delta+I$. The proof of Theorem \ref{thm3} and Theorem \ref{thm4} above employs new tools such as:
\begin{itemize}
\item  sharp integral estimates with kernel $K_{\alpha, \beta}$ (see Lemma \ref{lm_1} and \ref{lm_2} below);
\item a new sub and super-solution method (see Proposition \ref{ps}) which is different from those adopted in \cite{BC25, GMK17} that fits the non-local setting of the problem \eqref{eq0}. The $\mathcal{C}^2$-regularity of the limit solution is obtained through local estimates on Riesz potentials which we prove in Lemma \ref{rieszlog}. 
\end{itemize}    
\medskip
The remainder of this article is organized as follows. In Section 2 we provide various integral estimates that extend those obtained in \cite{FG20, G23, GKS21, GMT11, GT15, GT16, GY23}. In Section 3 we provide a property of the fundamental solution to $-\Delta+\mu I$, $\mu>0$. Section 4 is devoted to the local regularity estimates for Riesz potentials. Lemma \ref{rieszlog} extends the global estimates of the standard Riesz potentials discussed in \cite[Theorem 2.1, page 153]{M96}. Sections 5 and 6 are devoted to the proofs of Theorem \ref{thm1} and Theorem \ref{thm2} respectively. In Section 7 we devise a non-local version of the sub and super-solution method. This tool will be used in Sections 8 and 9 where we prove Theorem \ref{thm3} and Theorem \ref{thm4}.

\section{Integral estimates} 
\begin{lemma}\label{lm_1}
     Let $0 \leq \alpha < N$, $\beta \in \R$, $0\leq \rho\leq \frac{1}{3}$ and $\gamma \geq 0$. Define:
     \begin{equation}
         \Ix = \int_{\rho<|y|<1}\frac{\log^{\beta}\frac{2e}{|x-y|}}{|x-y|^{\alpha}|y|^{\gamma}}dy \, , \quad\mbox{ for all }x\in B_1\setminus B_\rho.
     \end{equation}
     \begin{enumerate}[(i)]
         \item If $\gamma \geq N$, then $\Ix = \infty$ for all $x \in B_1 \setminus \{0\}$.
         \item If $0 \leq \gamma < N$, then there exists a constant $C>0$ independent of $\rho$ such that  for all $x \in B_1\setminus B_\rho$, we have:
        \begin{equation}\label{ro1}
             \Ix \geq C \left\{
        \begin{aligned}
            &|x|^{N-\alpha-\gamma}\log^{\beta}\frac{2e}{|x|}\quad&\text{if }\alpha + \gamma > N&,\\[0.2cm]
            &\log^{1+\beta}\frac{2e}{|x|}\quad&\text{if }\alpha + \gamma = N&,\, \beta > -1,\\[0.2cm]
            &\log\left(\log\frac{2e}{|x|}\right)\quad&\text{if }\alpha + \gamma = N&,\, \beta = -1,\\[0.2cm]
            &1 \quad&\text{if }\alpha + \gamma = N&,\, \beta < -1,\\[0.2cm]
            &1 \quad&\text{if }\alpha + \gamma < N&.
        \end{aligned}\right.
        \end{equation}
         \item If $0 \leq \gamma < N$, then 
         \begin{equation}\label{lm1}
             I_{\alpha, \beta, \gamma}(0,x) \simeq \left\{
        \begin{aligned}
            &|x|^{N-\alpha-\gamma}\log^{\beta}\frac{2e}{|x|}\quad&\text{if }\alpha + \gamma > N&,\\[0.2cm]
            &\log^{1+\beta}\frac{2e}{|x|}\quad&\text{if }\alpha + \gamma = N&,\, \beta > -1,\\[0.2cm]
            &\log\left(\log\frac{2e}{|x|}\right)\quad&\text{if }\alpha + \gamma = N&,\, \beta = -1,\\[0.2cm]
            &1 \quad&\text{if }\alpha + \gamma = N&,\, \beta < -1,\\[0.2cm]
            &1 \quad&\text{if }\alpha + \gamma < N&.
        \end{aligned}\right.
         \end{equation}
     \end{enumerate}
\end{lemma}
\begin{proof}
(i) Assume $\gamma \geq N$. For all $0 < |x| < 1$, let $0 < |y| < \frac{|x|}{2}$, so that
   \[\frac{|x|}{2} < |x| - |y| \leq |x -y| \leq |x| + |y| < 2|x|.\]
   It follows that
   \[\log\frac{e}{|x|} < \log\frac{2e}{|x-y|} < \log\frac{4e}{|x|}.\]
   Thus,
   $$
   \begin{aligned}
       \Ix &\geq \left(2|x|\right)^{-\alpha} \min\left\{\log^{\beta}\frac{e}{|x|}, \log^{\beta}\frac{4e}{|x|}\right\}\int_{|y| < \frac{|x|}{2}}|y|^{-\gamma}dy \\[0.2cm]
       &= C|x|^{-\alpha}\min\left\{\log^{\beta}\frac{e}{|x|}, \log^{\beta}\frac{4e}{|x|}\right\}\int_0^{\frac{|x|}{2}}t^{N - \gamma}\frac{dt}{t} = \infty.
   \end{aligned}
   $$
   (ii) Assume $0 \leq \gamma < N$. Then, we notice that 
   $$
   \Ix \leq  \K1 * f,
   $$
   where $\K1$ is defined in \eqref{ptt} and $f(x)= |x|^{-\gamma}$. Since $0\leq \alpha, \gamma <N$, it follows that $\K1, f \in L^1(B_1)$, and hence $\Ix < \infty$ for all $x \in B_1\setminus B_\rho$. \\
It is enough to establish the inequality \eqref{ro1} for all $r<|x|<\frac{2}{5}$.  Since all functions involved in \eqref{ro1} are continuous in $B_1\setminus \overline{B}_{\rho}$, the inequality \eqref{ro1} extends to all $x \in B_1\setminus B_\rho$ by choosing a smaller constant $c > 0$ which is independent of $\rho$. \\
   Now, let $\rho < |x| < \frac{2}{5}$ and $|y| > 2|x|$. Then $\frac{|y|}{2} < |x-y| < 2|y|$. Thus,
   \begin{equation}\label{lbe}
   \begin{aligned}
       \Ix &\geq \int_{2|x| < |y| < 1}\frac{\min\left\{\log^{\beta}\frac{e}{|y|}, \log^{\beta}\frac{4e}{|y|}\right\}}{(2|y|)^{\alpha}|y|^{\gamma}}dy \\[0.2cm]
       &= c\int_{2|x|}^1\min\left\{\log^{\beta}\frac{e}{t}, \log^{\beta}\frac{4e}{t}\right\}t^{N - \alpha - \gamma} \frac{dt}{t}.
      \end{aligned}
      \end{equation}
   If $\alpha + \gamma > N$ and $\beta \leq 0$, then the mapping $t \mapsto \log^{\beta}\frac{4e}{t}$ is increasing. Thus, from \eqref{lbe}, we obtain:
   \begin{align*}
       \Ix \geq C\log^{\beta}\frac{2e}{|x|}\int_{2|x|}^1 t^{N - \alpha - \gamma} \frac{dt}{t} \geq C|x|^{N-\alpha-\gamma}\log^{\beta}\frac{2e}{|x|} \quad\text{for some } C > 0.
   \end{align*}
   If $\alpha + \gamma > N$ and $\beta > 0$, then from \eqref{lbe} we have
   \begin{align*}
        \Ix &\geq c\int_{2|x|}^{\frac{5|x|}{2}}t^{N-\alpha-\gamma}\log^{\beta}\frac{e}{t}\frac{dt}{t} \\
        &\geq C\log^{\beta}\frac{2e}{5|x|}\int_{2|x|}^{\frac{5|x|}{2}}t^{N - \alpha - \gamma} \frac{dt}{t} \\
        &\geq C|x|^{N-\alpha-\gamma}\log^{\beta}\frac{2e}{|x|} \quad\text{for some } C > 0.
   \end{align*}
   If $\alpha + \gamma < N$, we take $\varepsilon > 0$ and use the fact that
   \[\min\left\{\log^{\beta}\frac{e}{t}, \log^{\beta}\frac{4e}{t}\right\} \geq ct^{\varepsilon}\quad\text{for all } 0 < t < 1.\]
   Thus, \eqref{lbe} yields
   \begin{equation*}
       \Ix \geq c\int_{2|x|}^1 t^{N-\alpha-\gamma + \varepsilon} \frac{dt}{t} \geq C \quad\text{for some } C > 0.
   \end{equation*}
   If $\alpha + \gamma = N$, then from \eqref{lbe}, we have
   \begin{align*}
       \Ix &\geq c\int_{2|x|}^1 \min\left\{\log^{\beta}\frac{e}{t}, \log^{\beta}\frac{4e}{t}\right\}\frac{dt}{t} \\[0.2cm]
       &\geq C\left\{
       \begin{aligned}
           &1 \quad&\text{if }\beta < -1\\[0.2cm]
           &\log\left(\log\frac{2e}{|x|}\right) \quad&\text{if }\beta = -1\\[0.2cm]
           &\log^{1+\beta}\frac{e}{|x|} \quad&\text{if }\beta > -1
       \end{aligned}\right.\quad\text{for some } C > 0.
   \end{align*}
(iii) The lower bound in \eqref{lm1} was established in part (ii) above in which we take $\rho=0$. With the same argument as above, it is enough to establish the upper bound in \eqref{lm1} for $0<|x|\leq \frac{1}{3}$. Let $r = |x|$, so that $0<r\leq \frac{1}{3}$ and assume that $x = r\varsigma$ and $y = r\eta$, where, in particular, $|\varsigma| = 1$. Then we have:
   \begin{equation}\label{ub}
   \begin{aligned}
       I_{\alpha, \beta, \gamma}(0,x) &\leq \int_{|y| < 2}\frac{\log^{\beta}\frac{2e}{|x-y|}}{|x - y|^{\alpha}|y|^{\gamma}}dy \\
       &= \int_{r|\eta|<2}\frac{\log^{\beta}\frac{2e}{r|\varsigma - \eta|}}{|r|\varsigma| - r|\eta||^{\alpha}(r|\eta|)^{\gamma}}d(r\eta) \\
       &= r^{N-\alpha-\gamma}\left\{\int_{0 < |\eta| \leq 2} + \int_{2 < |\eta| < \frac{2}{r}}\right\}\frac{\log^{\beta}\frac{2e}{r|\varsigma - \eta|}}{|\varsigma-\eta|^{\alpha}|\eta|^{\gamma}}\,d\eta \\
       &=:  r^{N-\alpha-\gamma}(A + B).
   \end{aligned}
   \end{equation}
Next, we divide our proof into two cases.\medskip\\
{\bf Case 1}: $\beta \leq 0$. We begin by estimating the quantity $A$ in \eqref{ub}. Since $0 < |\eta| \leq 2$ and $|\varsigma| = 1$, it follows that $|\varsigma - \eta| \leq |\varsigma| + |\eta| \leq 3$ and then
   \begin{align}\label{ubA}
       A \leq \log^{\beta}\frac{2e}{3r}\int_{0 < |\eta| \leq 2}\frac{d\eta}{|\varsigma - \eta|^{\alpha}|\eta|^{\gamma}} \leq C\log^{\beta}\frac{2e}{3r} \leq C\log^{\beta}\frac{2e}{r}\quad \text{for some } C > 0,
   \end{align}
   since $0<r\leq 1/3$ and 
   $$0 < \max\limits_{|\varsigma| = 1}\int_{0 < |\eta| \leq 2}\frac{d\eta}{|\varsigma - \eta|^{\alpha}|\eta|^{\gamma}} < \infty.$$
   To estimate the quantity $B$ in \eqref{ub}, we observe that $\frac{|\eta|}{2} < |\varsigma - \eta| < 2|\eta|$, which allows us to use $\beta \leq 0$ to obtain:
   \begin{align}\label{ubB}
       B &\leq  C\int_{2<|\eta|<\frac{2}{r}}|\eta|^{-\alpha - \gamma} \log^{\beta}\frac{e}{r|\eta|} d\eta \nonumber \\
       &= C\int_2^{\frac{2}{r}}t^{N-\alpha-\gamma}\log^{\beta}\frac{e}{rt}\frac{dt}{t} \,,\quad \text{for some } C > 0.
   \end{align}
   If $\alpha + \gamma > N$, we take $0<\varepsilon<\alpha+\gamma-N$ small. The function 
   $$g:\Big[2,\frac{2}{r}\Big] \to (0, \infty)\, , \quad g(t)=t^{-\varepsilon}\log^\beta\frac{e}{rt}$$ 
   is decreasing for $\beta=0$ while for $\beta<0$ it satisfies 
   $$g'(t)=-\varepsilon t^{-\varepsilon-1}\log^{\beta-1}\frac{e}{rt}\Big\{\log\frac{e}{rt}+\frac{\beta}{\varepsilon}\Big\}.$$ 
   Thus $g$ has at most one critical point and for $\varepsilon>0$ small enough we have $g'\left(\frac{2}{r}\right)>0$. This means that $g$ achieves its maximum at the endpoints of the interval $\big[2,\frac{2}{r}\big]$. Note that 
   $$
   g(2)=2^{-\varepsilon}\log^\beta \frac{e}{2r} \quad\mbox{ and }\quad g\left(\frac{2}{r}\right)=2^{-\varepsilon} r^\varepsilon \log^\beta \frac{e}{2}.$$
   Hence, for some constant $C_\varepsilon>0$ we have
   $$
   g(t)\leq C_\varepsilon \log^\beta \frac{e}{2r} \quad \mbox{ for all } 2\leq t\leq \frac{2}{r} , \, 0<r\leq \frac{1}{3}.$$
   Using this last estimate in \eqref{ubB}, we derive
   \begin{equation}\label{ub>}
       B \leq C\log^{\beta}\frac{e}{2r}\int_2^{\frac{2}{r}}t^{N-\alpha-\gamma+\varepsilon}\,\frac{dt}{t} \leq C\log^{\beta}\frac{e}{2r} \leq C\log^{\beta}\frac{2e}{r} \quad\text{for some } C > 0.
   \end{equation}
   If $\alpha + \gamma < N$, then from \eqref{ubB} we have
   \begin{equation}\label{ub<}
       B \leq C\log^{\beta}\frac{e}{2}\int_2^{\frac{2}{r}}t^{N-\alpha-\gamma}\,\frac{dt}{t}  \leq Cr^{\alpha + \gamma - N} \quad\text{for some } C > 0.
   \end{equation}
   If $\alpha + \gamma = N$, then from \eqref{ubB} we have
    \begin{equation}\label{ub=}
    B \leq C\left\{\begin{aligned}
           &1 \quad&\text{if } \beta < -1\\[0.2cm]
           &\log\log\frac{e}{2r}\quad&\text{if } \beta = -1\\[0.2cm]
           &\log^{1+\beta}\frac{e}{2r}\quad&\text{if } \beta > -1\\
       \end{aligned}\right.\quad\text{for some } C > 0.
    \end{equation}
    Now, we combine \eqref{ub}, \eqref{ubA}, \eqref{ub>}, \eqref{ub<} and \eqref{ub=} to estimate:
    \begin{align}
        I_{\alpha, \beta, \gamma}(0,x) &\leq Cr^{N-\alpha-\gamma}\left\{\log^{\beta}\frac{2e}{r} + B\right\} \nonumber\\[0.2cm]
        &\leq C\left\{\begin{aligned}
            &1 \quad&\text{if } \alpha + \gamma < N&\\[0.2cm]
            &1 \quad&\text{if } \alpha + \gamma = N&, \beta < -1\\[0.1cm]
            &\log\log\frac{2e}{|x|} \quad&\text{if } \alpha + \gamma = N&, \beta = -1\\
            &\log^{1+\beta}\frac{2e}{|x|} \quad&\text{if } \alpha + \gamma = N&, \beta > -1\\
            &|x|^{N-\alpha-\gamma}\log^{\beta}\frac{2e}{|x|} \quad&\text{if } \alpha + \gamma > N&
        \end{aligned}\right.\quad\text{for some } C > 0.
    \end{align} 
{\bf Case 2: $\beta > 0$}. To estimate the integral $B$ in \eqref{ub}, we notice that as in Case 1 above we have $\frac{|\eta|}{2} < |\varsigma - \eta| < 2|\eta|$. Thus, we obtain
    \begin{align}\label{ubB2}
        B &\leq C\log^{\beta}\frac{4e}{r}\int_{2 < |\eta| < \frac{2}{r}}\frac{d\eta}{|\eta|^{\alpha + \gamma}} \nonumber\\
        &= C\log^{\beta}\frac{4e}{r} \int_2^{\frac{2}{r}} t^{N-\alpha - \gamma} \frac{dt}{t} \nonumber\\[0.2cm]
        &\leq C\log^{\beta}\frac{4e}{r}\left\{\begin{aligned}
            &1 \quad&\text{if } \alpha + \gamma > N\\
            &\log\frac{1}{r}&\text{if } \alpha + \gamma = N\\
            &r^{\alpha + \gamma - N}&\text{if } \alpha + \gamma < N\\
        \end{aligned}\right.\quad\text{for some } C > 0.
    \end{align}
    For the integral $A$ in \eqref{ub}, we note that $|\eta| \leq 2$ and $|\varsigma| = 1$ implies $|\varsigma-\eta| \leq 3$ and then
    \begin{align}
        A &\leq \int_{|\varsigma-\eta| \leq 3}\frac{\log^{\beta}\frac{2e}{r|\varsigma-\eta|}\,d\eta}{|\eta-\varsigma|^{\alpha}|\eta - \varsigma+\varsigma|^{\gamma}}\nonumber\\
        &= \int_{|z|\leq 3}\frac{\log^{\beta}\frac{2e}{r|z|}}{|z|^{\alpha}|z + \varsigma|^{\gamma}}dz  \nonumber \\
        &= \int_{|z| \leq \frac{1}{2}}\frac{\log^{\beta}\frac{2e}{r|z|}}{|z|^{\alpha}|z + \varsigma|^{\gamma}}dz + \int_{\frac{1}{2} < |z| \leq 3} \frac{\log^{\beta}\frac{2e}{r|z|}}{|z|^{\alpha}|z + \varsigma|^{\gamma}}dz =: A_1 + A_2. \nonumber
    \end{align}
    To estimate $A_1$, we note that $|z + \varsigma| \geq |\varsigma| - |z| = 1 - |z|$ and thus
    \begin{align}\label{ubA_1}
        A_1 &\leq C\int_0^{\frac{1}{2}}t^{N-\alpha}(1-t)^{-\gamma}\log^{\beta}\frac{2e}{rt}\frac{dt}{t} \nonumber\\
        &\leq C2^{\gamma}\int_0^{\frac{1}{2}}t^{N-\alpha}\log^{\beta}\frac{2e}{r t}\frac{dt}{t}.
    \end{align}
    Take $0 < \varepsilon < N - \alpha$ and note that for $C > 1 + \frac{\beta}{\varepsilon}$ large, the function
    $$h(t) = \log^{\beta}\frac{2e}{rt} - Ct^{-\varepsilon}\log^{\beta}\frac{4e}{r}$$
    is increasing for all $t \in (0, \frac{1}{2}]$ since
    $$h'(t) = -\frac{\beta}{t}\log^{\beta-1}\frac{2e}{rt} + C\varepsilon t^{-\varepsilon-1}\log^{\beta}\frac{4e}{r} > 0\quad\text{for all } 0 < t \leq \frac{1}{2}.$$
    It follows that $h(t) \leq h(\frac{1}{2}) < 0$ and
    $$\log^{\beta}\frac{2e}{rt} \leq Ct^{-\varepsilon}\log^{\beta}\frac{4e}{r} \quad\text{for all } t \in (0, \frac{1}{2}]\, , r \in (0,1).$$
    Thus, from \eqref{ubA_1} we estimate:
    \begin{align}\label{ubA_12}
        A_1 \leq C\log^{\beta}\frac{4e}{r}\int_0^{\frac{1}{2}}t^{N - \alpha - \varepsilon} \frac{dt}{t} \leq C\log^{\beta}\frac{4e}{r}.
    \end{align}
    Moreover, we have
    \begin{align}\label{ubA_2}
        A_2 &\leq C\log^{\beta}\frac{4e}{r}\int_{\frac{1}{2}< |z| \leq 3}|z|^{-\alpha}|z + \varsigma|^{-\gamma}dz \nonumber \\
        &\leq C\log^{\beta}\frac{4e}{r}\int_{\frac{1}{2}< |z| \leq 3}|z + \varsigma|^{-\gamma}dz \nonumber \\
        &\leq C\log^{\beta}\frac{4e}{r} \int_{|z + \varsigma| \leq 4} |z + \varsigma|^{-\gamma}dz \nonumber \\
        &=  C\log^{\beta}\frac{4e}{r} \int_0^4 t^{N - \gamma} \frac{dt}{t} \leq C\log^{\beta}\frac{4e}{r}.
    \end{align}
    Hence, combining \eqref{ub}, \eqref{ubB2}, \eqref{ubA_12} and \eqref{ubA_2}, we obtain the upper bound in \eqref{lm1} for the case $\beta > 0$. Together with the case $\beta \leq 0$, this concludes the proof of the upper bound and thereby completes the proof of our Lemma.
\end{proof}
\begin{lemma}\label{lm_2}
Let $0 \leq \alpha < N$, $\beta \in \R$, $0\leq \rho\leq \frac{1}{3}$ and $\theta \geq 0$. Define:
\begin{equation}
    \Jx = \int_{\rho<|y|<1}\frac{\log^{\beta}\frac{2e}{|x-y|}\log^{\theta}\frac{2e}{|y|}}{|x-y|^{\alpha}}dy\, , \quad\mbox{ for all }x\in B_1\setminus B_\rho.
\end{equation}
Then, there exist two constants $c_2>c_1>0$ independent of $\rho$ such that 
\begin{equation}\label{jro}
c_1\leq \Jx\leq c_2\quad\mbox{ for all }\rho<|x|<1.
\end{equation}
\end{lemma}
\begin{proof}
    Let $\rho<|x|<1$. We first derive the lower bound in \eqref{jro}. Taking $0 < \varepsilon < \alpha$,  we have
    \begin{align}\label{jlb}
    \Jx &\geq \int_{\frac{1}{3}<|y| < 1} \frac{\log^{\beta}\frac{2e}{|x - y|}}{|x - y|^{\alpha}}dy \nonumber\\
    &\geq C\int_{\frac{1}{3}< |y| < 1} |x - y|^{\varepsilon - \alpha} dy \nonumber\\
    &\geq C\int_{\frac{1}{3}< |y| < 1} (1 + |y|)^{\varepsilon - \alpha} dy \geq C > 0.
    \end{align}
    For the upper bound, we decompose $\Jx$ into the following two parts: 
    \begin{equation}\label{jub}
        \Jx \leq  \left\{\int_{0<|y| \leq \frac{|x|}{2}} + \int_{\frac{|x|}{2} < |y| < 1}\right\}\frac{\log^{\beta}\frac{2e}{|x - y|}\log^{\theta}\frac{2e}{|y|}}{|x - y|^{\alpha}}dy =: J_1 + J_2.
    \end{equation}
    To estimate $J_1$ in \eqref{jub}, let $0 < \varepsilon < N - \alpha$ be small. For $0<|y| \leq \frac{|x|}{2}$, we observe that $\frac{|x|}{2} < |x| - |y| \leq |x - y| \leq \frac{3|x|}{2}$. Then,
    \begin{align}\label{j1}
        J_1 &\leq C\int_{|y| \leq \frac{|x|}{2}}|x|^{-\alpha}\max\left\{\log^{\beta}\frac{4e}{|x|}, \log^{\beta}\frac{4e}{3|x|} \right\} \log^{\theta}\frac{2e}{|y|}\,dy \nonumber\\
        &= C|x|^{-\alpha}\max\left\{\log^{\beta}\frac{4e}{|x|}, \log^{\beta}\frac{4e}{3|x|} \right\}\int_0^{\frac{|x|}{2}}t^{N}\log^{\theta}\frac{2e}{t}\frac{dt}{t} \nonumber \\
        &\leq C|x|^{-\alpha}\max\left\{\log^{\beta}\frac{4e}{|x|}, \log^{\beta}\frac{4e}{3|x|} \right\}\int_0^{\frac{|x|}{2}}t^{N - \varepsilon} \frac{dt}{t} \leq C < \infty.
    \end{align}
    To estimate $J_2$ in \eqref{jub}, we note that $\frac{|x|}{2} < |y| < 1$. We consider the following two subcases:\\
    If $\frac{|x|}{2} < |y| \leq 2|x|$, then 
    $$|x - y| \leq |x| + |y| \leq 3|x| \leq 6|y| \quad \text{and thus} \quad \log^{\theta}\frac{2e}{|y|} \leq \log^{\theta}\frac{12e}{|x - y|}.$$
    If $2|x| < |y| < 1$, then 
    $$|x - y| \leq |x| + |y| \leq \frac{3|y|}{2}\quad \text{and thus} \quad \log^{\theta}\frac{2e}{|y|} \leq \log^{\theta}\frac{3e}{|x - y|}.$$
    It follows that
    \begin{align}\label{j2}
        J_2 &\leq \int_{\frac{|x|}{2} < |y| < 1}\frac{\log^{\beta}\frac{2e}{|x - y|}\log^{\theta}\frac{12e}{|x - y|}}{|x - y|^{\alpha}}dy \nonumber\\
        &\leq \int_{|z| < 2}\frac{\log^{\beta}\frac{2e}{|z|}\log^{\theta}\frac{12e}{|z|}}{|z|^{\alpha}}dz \quad\mbox{ where } z=x-y, |z|<2,\nonumber\\
        &= \int_0^2 t^{N - \alpha}\log^{\beta}\frac{2e}{t}\log^{\theta}\frac{12e}{t} \frac{dt}{t} \leq C < \infty.
    \end{align}
    Hence, by combining \eqref{jlb}, \eqref{jub}, \eqref{j1} and \eqref{j2}, we conclude the proof of the lemma.
\end{proof}

\section{The fundamental solution of \texorpdfstring{$-\Delta+\mu I$}{-∆ + μI}}
Let $\mu>0$ be a positive real number. The fundamental solution of the operator $-\Delta+\mu I$ in $\R^N$, $N\geq 3$, is given by (see \cite[Section 3]{B96})  
\begin{equation}\label{fbesselp}
G_{\mu}(x)=\frac{|x|^{2-N}}{(2\pi)^{N/2}}  \frac{K_{\frac{N-2}{2}}(\sqrt \mu |x|)}{(\sqrt \mu |x|)^{-\frac{N-2}{2}}},
\end{equation}
where $K_{\frac{N-2}{2}}$ is the modified Bessel function of the second kind of order $\frac{N-2}{2}$. Precisely, we have
$$
(-\Delta+\mu I) G_{\mu}=\delta_0\quad\mbox{ in }\mathscr{D}'(\R^N), N\geq 3.
$$
From \cite[page 415]{AS61} we have the following property of $K_{\frac{N-2}{2}}$:
\begin{equation}\label{Be2}
\lim_{z\to 0} \frac{K_{\frac{N-2}{2}}(z)}{|z|^{-\frac{N-2}{2}}}=2^{\frac{N-4}{2}}\Gamma\left(\frac{N-2}{2}\right),
\end{equation}
where $\Gamma$ is the standard Gamma function.

\begin{lemma}\label{fs}
Let $N\geq 3$. Then, for all $\mu>0$ we have 
$$\lim_{|x|\to 0}\frac{G_\mu(x)}{E(|x|)}=1,
$$
where $E(|x|)$ is the fundamental solution of the Laplace operator, see \eqref{E}.
\end{lemma}
\begin{proof} Since $\sigma_N=\frac{2\pi^{N/2}}{\Gamma\left(\frac{N}{2}\right)}$, from \eqref{E} and \eqref{Be2} we find
$$
\lim_{|x|\to 0}\frac{G_{\mu}(x)}{E(|x|)}= \frac{N-2}{2^{\frac{N-2}{2}}\Gamma\left(\frac{N}{2}\right)} \lim_{|x|\to 0} \frac{K_{\frac{N-2}{2}}(\sqrt \mu |x|)}{(\sqrt \mu |x|)^{-\frac{N-2}{2}}}=
\frac{\frac{N-2}{2}\Gamma\left(\frac{N-2}{2}\right)}{\Gamma\left(\frac{N}{2}\right)}=1.
$$
\end{proof}

\section{Local regularity of Riesz potentials}

\noindent Let $f\in L^1(B_1)$ and $0\leq \alpha<N$, $\beta\in \R$. For all $x\in B_1$ we set
$$I_{\alpha, \beta} f(x):=\left(K_{\alpha, \beta}*f\right)(x)=\int_{B_1}\frac{ \log^{\beta} \frac{2e}{|x-y|} }{|x-y|^\alpha} f(y)dy.$$
\begin{lemma}\label{rieszlog}
  Let $B\subset B_1$ be an open ball and $\rho:={\rm dist}(0, B)/2\geq 0$. Let also $1<s<\infty$ be such that $\alpha<\frac{N}{s'}<\alpha+1$, where $s'=\frac{s}{s-1}>1$.
\begin{enumerate}
\item[(i)] If $\rho>0$ and $f\in L^1(B_1)\cap L^s(B_1\setminus B_\rho)$, then, 
there exists a constant $C=C(N, \alpha, \beta, \rho, s)>0$ such that
\begin{equation}\label{rl1}
    |I_{\alpha, \beta}f(x)-I_{\alpha, \beta}f(y)|\leq C|x-y|^{\frac{N}{s'}-\alpha}\log^{\beta^+}\left(\frac{2e}{|x-y|}\right)\left( \|f\|_{L^s(B_1\setminus B_\rho)}+ \|f\|_{L^1(B_\rho)} \right),
\end{equation}
for all $x, y\in B$, where $\beta^+=\max\{\beta, 0\}$.
\item[(ii)] If $f\in L^s(B_1)$, then, 
there exists a constant $C=C(N, \alpha, \beta, s)>0$ such that
\begin{equation}\label{rl2}
    |I_{\alpha, \beta}f(x)-I_{\alpha, \beta}f(y)|\leq C|x-y|^{\frac{N}{s'}-\alpha}\log^{\beta^+}\left(\frac{2e}{|x-y|}\right) \|f\|_{L^s(B_1)} ,
\end{equation}
for all $x, y\in B_1$. In particular, for all $0<\varepsilon<\frac{N}{s'}-\alpha$, the mapping
$$
I_{\alpha, \beta}: L^s(B_1)\to \mathcal{C}^{0, \frac{N}{s'}-\alpha-\varepsilon}\left(\overline{B}_1\right) \quad\mbox{ is continuous}.$$
\end{enumerate}
\end{lemma}
\begin{proof}
(i) Let $x,y\in B$, $x\neq y$. We decompose
$$
I_{\alpha, \beta} f(x)=\Big\{\int\limits_{|z|<\rho}+\int\limits_{\substack{|z|>\rho\\ |z-x|<2|x-y|}}+\int\limits_{\substack{|z|>\rho\\ |z-x|\geq 2|x-y|}} \Big\} K_{\alpha, \beta}(z-x)f(z) dz,
$$
$$
I_{\alpha, \beta} f(y)=\Big\{\int\limits_{|z|<\rho}+\int\limits_{\substack{|z|>\rho\\ |z-x|<2|x-y|}}+\int\limits_{\substack{|z|>\rho\\ |z-x|\geq 2|x-y|}} \Big\} K_{\alpha, \beta}(z-y)f(z) dz,
$$
so that we estimate
\begin{equation}\label{rr0}
\big|I_{\alpha, \beta} f(x)-I_{\alpha, \beta} f(y)  \big| \leq A_1+A_2+A_3+A_4,
\end{equation}
where
$$
\begin{aligned}
A_1&= \int\limits_{\substack{|z|>\rho\\ |z-x|<2|x-y|}} K_{\alpha, \beta}(z-x) |f(z)| dz,\\[0.1in]
A_2&= \int\limits_{\substack{|z|>\rho\\ |z-x|<2|x-y|}} K_{\alpha, \beta}(z-y) |f(z)| dz,\\[0.1in]
A_3&= \int\limits_{\substack{|z|>\rho\\ |z-x|\geq 2|x-y|}} \big| K_{\alpha, \beta}(z-x)-K_{\alpha, \beta}(z-y)\big|  |f(z)| dz,\\[0.2in]
A_4&= \int\limits_{|z|<\rho} \big| K_{\alpha, \beta}(z-x)-K_{\alpha, \beta}(z-y)\big|  |f(z)| dz.
\end{aligned}
$$
To estimate $A_1$ we use H\"older's inequality and we find
\begin{equation}\label{rr1}
\begin{aligned}
A_1&\leq \left( \,\int\limits_{|z-x|<2|x-y|} |z-x|^{-\alpha s'}\log^{\beta^+ s'}\left(\frac{2e}{|z-x|}\right) dz \right)^{\frac{1}{s'}}\|f\|_{L^s(B_1\setminus B_\rho)}\\[0.1in]
&=  C \left( \int\limits_0^{2|x-y|} t^{N-\alpha s'} \log^{\beta^+ s'}\left(\frac{2e}{t}\right) \frac{dt}{t}\right)^{\frac{1}{s'}}\|f\|_{L^s(B_1\setminus B_\rho)}\\[0.1in]
&\leq  C |x-y|^{\frac{N}{s'}-\alpha } \log^{\beta^+ }\left(\frac{2e}{|x-y|}\right) \|f\|_{L^s(B_1\setminus B_\rho)}.
\end{aligned}
\end{equation}
With the same method and using the fact that
$$
|z-x|<2|x-y|\Longrightarrow |z-y|\leq |z-x|+|x-y|<3|x-y|,
$$
we have 
\begin{equation}\label{rr2}
\begin{aligned}
A_2&\leq \int\limits_{|z-y|<3|x-y|} K_{\alpha, \beta}(z-y) |f(z)| dz \\[0.1cm]
& \leq  C |x-y|^{\frac{N}{s'}-\alpha } \log^{\beta^+ }\left(\frac{2e}{|x-y|}\right) \|f\|_{L^s(B_1\setminus B_\rho)}.
\end{aligned}
\end{equation}
Next, we estimate $A_3$. By the Mean Value Theorem there exists $\zeta\in [x, y]\subset B$ such that 
$$
\big| K_{\alpha, \beta}(z-x)-K_{\alpha, \beta}(z-y)\big|= |x-y|  \big| \nabla K_{\alpha, \beta}(z-\zeta)\big|.
$$
Note that 
$$
\big| \nabla K_{\alpha, \beta}(z-\zeta)\big| \leq C|z-\zeta|^{-\alpha-1}\log^{\beta^+}\frac{2e}{|z-\zeta|},
$$
and from $\zeta\in [x, y]$ and $|z-x|\geq 2|x-y|$ one gets
$$
|z-\zeta|\geq |z-x|-|x-\zeta|\geq |z-x|-|x-y|\geq \frac{|z-x|}{2}.
$$
Hence,
$$
\big| \nabla K_{\alpha, \beta}(z-\zeta)\big| \leq C|z-x|^{-\alpha-1}\log^{\beta^+}\frac{4e}{|z-x|}.
$$
By H\"older's inequality and the fact that $\frac{N}{s'}<\alpha+1$ we estimate
\begin{equation}\label{rr3}
\begin{aligned}
A_3&\leq |x-y| \int\limits_{\substack{|z|>\rho\\ |z-x|\geq 2|x-y|}}  
|z-x|^{-\alpha-1}\log^{\beta^+}\frac{4e}{|z-x|}|f(z)| dz\\[0.1in]
&\leq  C |x-y| \left(\, \int\limits_{2|x-y|}^\infty  t^{N-(\alpha+1) s'} \log^{\beta^+ s'}\left(\frac{4e}{t}\right) \frac{dt}{t}\right)^{\frac{1}{s'}}\|f\|_{L^s(B_1\setminus B_\rho)}\\[0.1in]
&\leq  C |x-y|^{\frac{N}{s'}-\alpha } \log^{\beta^+ }\left(\frac{2e}{|x-y|}\right) \|f\|_{L^s(B_1\setminus B_\rho)}.
\end{aligned}
\end{equation}
Finally, to estimate $A_4$ we use again the Mean Value Theorem and obtain (see Figure \ref{f1} below):
$$
\big| K_{\alpha, \beta}(z-x)-K_{\alpha, \beta}(z-y)\big|= |x-y|  \big| \nabla K_{\alpha, \beta}(z-\xi)\big|   \quad\mbox{ for some } \xi\in [x, y]\subset B.
$$
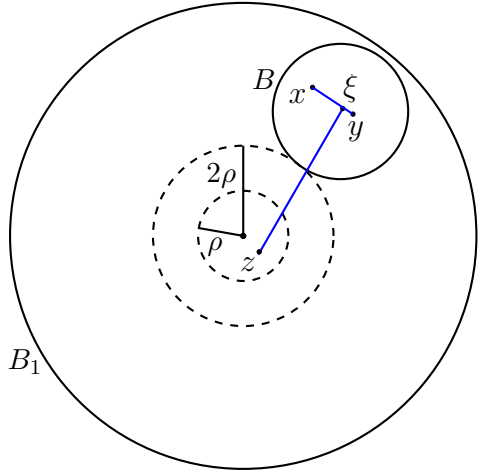
\begin{figure}[ht]
\begin{center}
\begin{tikzpicture}[scale=1.0]
    \draw[thick] (0,0) circle (3.1cm);

    \draw[dashed,thick] (0,0) circle (0.6cm);
    \draw[dashed,thick] (0,0) circle (1.2cm);
    \filldraw (0,0) circle (1pt);
    \draw[thick] (0,0) -- (90:1.2cm);
    \draw[thick] (0,0) -- (170:0.6cm);
    \node at (110:0.85cm) {\small $2\rho$};
    \node at (200:0.4cm) {\small $\rho$};
    \draw[thick] (52:2.1cm) circle (0.9cm);
    \draw[color=blue,thick] (0.2,-0.25) -- (52:2.15cm);
    \node at (100:-0.38cm) {$z$};
    \node at (82.5:2.1cm) {\small $B$};
    \node at (68.75:2.0cm) {$x$};
    \node at (43:2.05cm) { $y$};
    \node at (54:2.43cm) { $\xi$};
    \node at (-150:3.35cm) {\small $B_1$};
    \filldraw (48:2.18cm) circle (0.8pt);
    \filldraw (52:2.15cm) circle (0.8pt);
    \filldraw (65:2.18cm) circle (0.8pt);
    \filldraw (135:-0.3cm) circle (0.8pt);
    \draw[color=blue,thick] (48:2.18cm) -- (65:2.18cm);
\end{tikzpicture}
\end{center}
    \caption{The ball $B\subset B_1$ and $\xi\in [x,y]$.}
    \label{f1}
\end{figure}
Since $\xi\in B$ and $|\xi|>2\rho$ one has
$$
|z|<\rho<2\rho<|\xi|\Longrightarrow |z-\xi|\geq |\xi|-|z|>\rho,
$$
and then
$$
\big| \nabla K_{\alpha, \beta}(z-\xi)\big| \leq C|z-\xi|^{-\alpha-1}\log^{\beta^+}\frac{2e}{|z-\xi|}\leq C(\rho)<\infty. 
$$
This yields
\begin{equation}\label{rr4}
\begin{aligned}
A_4&\leq |x-y| \|f\|_{L^1( B_\rho)}.
\end{aligned}
\end{equation}
Using now the estimates \eqref{rr1}-\eqref{rr4} into \eqref{rr0} we deduce \eqref{rl1}.\\
(ii) We apply part (i) above for $B=B_1$ and $\rho=0$.
\end{proof}

\section{Proof of Theorem \ref{thm1}}
Let $0\leq u\in \mathcal{C}^2(B_1\setminus\{0\})$ be a solution of \eqref{h1}. 
Taking the spherical average in \eqref{h1} we obtain:
\[\Delta \overline{u} = \frac{1}{r^{N - 1}}(r^{N - 1}\overline{u}')' \leq V(r)\overline{u} - \overline f\quad \text{for all } 0 < r < 1.\]
It follows that 
\begin{equation}\label{test1}
    (r^{N-1}\overline{u}')' \leq r^{N-1}V(r) \overline{u} - r^{N-1}\overline{f} \quad \text{for all } 0 < r < 1.
\end{equation}
We take $0 < R < 1$ small such that
\begin{equation}\label{dini0}
\left\{\begin{aligned}
    &\int_0^R s\log\frac{1}{s}V(s) ds < \frac{1}{2}\quad&\text{if }N = 2,\\[0.1cm]
    &\int_0^R sV(s) ds <  \frac{1}{2} \quad&\text{if }N \geq 3.
\end{aligned}\right.
\end{equation}
For all $0 < r \leq  R$, we define
$$g_R(r) = \int_{B_R \setminus B_r} V(|x|)u(x)dx = \sigma_N\int_r^{R}s^{N - 1}V(s)\overline{u}(s)ds.$$
Integrating \eqref{test1} over $[r, R]$ we find
$$
\begin{aligned}
    -r^{N-1}\overline{u}' &\leq \int_r^{R}s^{N - 1}V(s)\overline{u}(s)ds - \int_r^R s^{N-1}\overline{f}(s) ds + c \\[0.2cm]
    &\leq  g_R(r) - \int_r^R s^{N-1}\overline{f}(s) ds + c,
\end{aligned}
$$
where $c > 0$. Since $f \in L^1_{loc}(B_1)$,  for all $0 < r \leq R$, we have $\int_r^R s^{N-1}\overline{f}(s) ds < \infty$. Thus, there exists a constant $C > 0$ such that 
\begin{equation}\label{test2}
     -r^{N-1}\overline{u}' \leq g_R(r) + C \quad\text{for all } 0 < r \leq R.
\end{equation}
Then, for all $0 < r \leq R$ we have
\begin{equation}\label{test3}
  -\overline{u}'(r) \leq r^{1-N}g_R(r) + Cr^{1 - N}.
\end{equation} 
We use the fact that $g_R(r)$ is a non-increasing function since 
$$
\frac{dg_R(r)}{dr} = -r^{N-1}V(r)\overline{u}(r) \leq 0.
$$
Integrating \eqref{test3} over $[r, R]$ again, we deduce:
\begin{align}\label{test4}
   \overline{u}(r) &\leq \int_r^{R}s^{1-N}g_R(s)ds + C\int_r^{R}s^{1-N}ds+c \nonumber\\[0.2cm]
   &\leq g_R(r)\int_r^{R}s^{1-N}ds + C\int_r^{R}s^{1-N}ds+c \nonumber\\[0.2cm]
   &\leq \left\{
     \begin{aligned}
      &\left(\log\frac{1}{r}\right)g_R(r) + C\log\frac{1}{r} &\quad\text{if } N = 2,\\[0.2cm]
      &\frac{r^{2-N}}{N - 2}g_R(r) + Cr^{2-N} &\quad\text{if } N \geq 3.
  \end{aligned}\right.
\end{align}
It follows that
\begin{align*}
rV(r)\overline{u}(r) &\leq g_R(r) r\left(\log\frac{1}{r}\right)V(r) + Cr\left(\log\frac{1}{r}\right)V(r) &\text{if } N = 2,\\
r^{N-1}V(r)\overline{u}(r) &\leq \frac{g_R(r)}{N - 2}rV(r) + CrV(r) &\text{if } N \geq 3.
\end{align*}
Thus, by integrating the above inequalities over $[r, R]$, we estimate:
\begin{equation}\label{test5}
    g_R(r) \leq \left\{
    \begin{aligned}
        &(g_R(r) + C)\int_r^{R}s\log\frac{1}{s}V(s)ds &\text{if } N = 2,\\[0.2cm]
        &\left(\frac{g_R(r)}{N - 2} + C\right)\int_r^{R}sV(s)ds &\text{if } N \geq 3.
    \end{aligned}\right.
\end{equation}
By \eqref{dini0}, we conclude that there exists constants $C_1, C_2 > 0$ such that 
\[g_R(r) \leq \frac{1}{2} g_R(r) + C_1 \quad\text{for all }  0 < r \leq R.\]
It follows that $g_R$ is finite, so $V(|x|)u\in L^1_{loc}(B_1)$. From \eqref{test4} we deduce $u \in L^1_{loc}(\B1)$ and
\begin{equation}\label{test6}
    \overline{u}(r) \leq \left\{
    \begin{aligned}
          &C\log \frac{1}{r}\quad &\text{if } N = 2,\\
          &C r^{2-N}\quad &\text{if } N \geq 3.
    \end{aligned}\right.
\end{equation}
Let $F:=-V(|x|)u+f$. Then $F\in L^1_{loc}(B_1)$ and $-\Delta u\geq F$ in $B_1\setminus\{0\}$. Using Theorem A, we deduce that $\Delta u\in L^1_{loc}(B_1)$ and $u$ satisfies \eqref{h} for some $h\in L^1_{loc}(B_1)$ and $c\geq 0$.
\medskip\\
\noindent{\bf Proof of Corollary \ref{cor1}}
 Let $V(|x|) = |x|^{-\gamma}\log^{\tau}\frac{e}{|x|}$. In light of Corollary \ref{cor0}, our conclusion is equivalent to:
\begin{equation}\label{cor0eq}
    \left\{\begin{aligned}
    &\int_0^1 s\log\frac{1}{s}V(s) ds = c \int_0^1 s^{1-\gamma}\log^{1 + \tau}\frac{e}{s} ds < \infty \quad&\text{if }N = 2,\\
    &\int_0^1 sV(s) ds = c\int_0^1 s^{1-\gamma}\log^{\tau}\frac{e}{s} ds < \infty \quad&\text{if }N \geq 3.
    \end{aligned}\right.
\end{equation} 
The dominant factor near $s = 0^+$ of \eqref{cor0eq} is $s^{1-\gamma}$. If $\gamma>2$ the integrals in \eqref{cor0eq} are divergent. If $\gamma < 2$, it is easy to see that the integrals in \eqref{cor0eq} are convergent. In the case $\gamma = 2$, let $t = \log\frac{e}{s}$, then $ds = -e^{1-t}dt$. From \eqref{cor0eq}, we have:
\begin{equation*}
\left\{\begin{aligned}
    &\int_0^1 s^{1-\gamma}\log^{1+\tau}\frac{e}{s} ds = e\int_{1}^{\infty}t^{1+\tau} dt < \infty &\quad\text{if }N = 2, \gamma = 2&\text{ and }\tau < -2,\\
    &\int_0^1 s^{1-\gamma}\log^{\tau}\frac{e}{s} ds = e\int_{1}^{\infty} t^{\tau} dt < \infty &\quad\text{if }N \geq 3, \gamma = 2&\text{ and }\tau < -1.
\end{aligned}\right.
\end{equation*}
This concludes the proof.\qed

\section{Proof of Theorem \ref{thm2}}

The proof of Theorem \ref{thm2} is essentially based on the following result.
\begin{lemma}\label{lm_3}
    Let $N \geq 2$ and $g \in \mathcal{C}(B_1 \setminus \{0\})$ with 
    \begin{equation}\label{gint}
    \int_{B_{\frac{1}{2}}} g(x) dx = \infty.
    \end{equation}
    Then, the inequality
    \begin{equation}\label{lm_2iq0}
        -\Delta u \geq g \quad\text{in } B_1\setminus\{0\}
    \end{equation}
    has no solutions $0\leq u \in \mathcal{C}^2(B_1\setminus\{0\})$.
\end{lemma}
\begin{proof}
    Assume by contradiction that \eqref{gint} holds and that \eqref{lm_2iq0} has a solution $u\geq 0$. Taking the spherical average, we have:
    \[-\Delta \overline{u} \geq \overline{g}(r)\quad \text{for all } 0 < r < 1.\]
    It follows that
    \begin{align*}
        -(r^{N-1}\overline{u}')' \geq r^{N - 1}\overline{g}(r) \quad \text{for all } 0 < r < 1.  
    \end{align*}
    Fix $0 < R \leq \frac{1}{2}$ small. For all $0 < r < R$, integrate the above inequality over $[r, R]$ to obtain: 
    \begin{equation*}
        r^{N-1}\bar u'(r) \geq \int_r^R s^{N-1}\overline{g}(s) ds -c_1 \quad\text{for some } c_1 > 0.
    \end{equation*}
    It follows that
    \begin{equation}\label{lVm0}
    \overline{u}'(r) \geq r^{1-N}\int_r^R s^{N-1}\overline{g}(s) ds - c_1 r^{1-N} \quad\text{for all } 0 < r \leq R.
    \end{equation}
    By integrating \eqref{lVm0} over $[r,R]$ again, we divide our proof into two cases.\\
    In the case $N = 2$, we obtain:
    \begin{align}\label{lVN2}
        \overline{u}(R) \geq \overline{u}(R) - \overline{u}(r) &\geq c_1(\log r - \log R) + \int_r^R s^{-1}\int_s^R t \overline{g}(t)dt ds \nonumber\\[0.1in]
        &= C_1 - c_1\log \frac{1}{r} + \int_r^R s^{-1}\int_s^R t \overline{g}(t)dt ds \nonumber\\[0.1in]
        & = C_1 + \log \frac{1}{r} \left\{\frac{\displaystyle \int_r^R s^{-1}\int_s^R t \overline{g}(t)dt ds}{-\log r} - c_1\right\}.
    \end{align}
    By l'H\^{o}pital's rule, we have
    $$
    \begin{aligned}
    \lim\limits_{r \to 0}\frac{\displaystyle\int_r^R s^{-1}\int_s^R t \overline{g}(t)dt ds}{-\log r} & = \lim\limits_{r \to 0}\frac{\displaystyle -r^{-1}\int_r^R s \overline{g}(s)ds}{-r^{-1}}   \\[0.2cm]
    &= \frac{1}{2\pi} \lim\limits_{r \to 0}\int_{r < |x| < R}g(x) dx = \infty.
    \end{aligned}
    $$
    Thus, passing to the limit with $r\to 0$ in \eqref{lVN2} we find $\overline u(R)=\infty$, contradiction. \\
    In the case $N \geq 3$, we obtain from \eqref{lVm0} that:
    \begin{align}\label{lVN3}
        \overline{u}(R) \geq \overline{u}(R) - \overline{u}(r) &\geq \frac{c_1}{N-2}(R^{2-N} - r^{2-N}) + \int_r^R s^{1-N}\int_s^R t^{N-1}\overline{g}(t)dt ds \nonumber\\[0.1in]
        &= C_1 - C_2r^{2-N} + \int_r^R s^{1-N}\int_s^R t^{N-1}\overline{g}(t)dt ds \nonumber\\[0.1in]
        & = C_1 + r^{2-N}\left\{\frac{\displaystyle \int_r^R s^{1-N}\int_s^R t^{N-1}\overline{g}(t)dt ds}{r^{2-N}} -C_2\right\}.
    \end{align}
    By l'H\^{o}pital's rule, we have 
    $$
    \begin{aligned}
    \lim\limits_{r \to 0}\frac{\displaystyle \int_r^R s^{1-N}\int_s^R t^{N-1}\overline{g}(t)dt ds}{r^{2-N}} & = \lim\limits_{r \to 0}\frac{\displaystyle -r^{1-N}\int_r^R s^{N-1}\overline{g}(s)ds}{(2-N)r^{1-N}}\\[0.2cm]
    &= \frac{1}{(N-2)\sigma_N} \lim\limits_{r \to 0}\int_{r < |x| < R}g(x) dx = \infty.
    \end{aligned}
    $$
    To raise a contradiction, it remains to let $r\to 0$ in \eqref{lVN3} to derive $\overline u(R)=\infty$. 
    Hence, the inequality \eqref{lm_2iq0} has no solutions $0\leq u \in \mathcal{C}^2(B_1 \setminus \{0\})$.
\end{proof}
\begin{cor} 
Let $V, f\in \mathcal{C}( B_1\setminus\{0\})$ be such that $\int_{B_{\frac{1}{2}}}f(x) dx=\infty$. Then, the inequality 
    \begin{equation}\label{coriq0}
        -\Delta u + V(x) u \geq f \quad\text{in } B_1\setminus\{0\},
    \end{equation}
    has no solutions $0\leq u\in \mathcal{C}^2(B_1\setminus \{0\})$ with the property $V u \in L^1_{loc}(B_1)$.
\end{cor}
\begin{proof} Assume by contradiction that such a solution exists. 
    We apply Lemma \ref{lm_3} for $g(x) = f(x) - V(x) u(x)$. Since $V u \in L^1_{loc}(B_1)$ and $\int_{B_{\frac{1}{2}}}f(x) dx=\infty$, it follows that
    \[\int_{B_{\frac{1}{2}}}g(x) = \int_{B_{\frac{1}{2}}} f(x)dx - \int_{B_{\frac{1}{2}}}V(x)u(x)dx = \infty.\]
    Therefore, we raise a contradiction and conclude the proof.
\end{proof}
\noindent{\bf Proof of Theorem \ref{thm2} completed.} Let $g=(\K1*u^p)u^q-Vu$. Since $Vu\in L^1_{loc}(B_1)$ we have either $g\in L^1_{loc}(B_1)$ or $\int_{B_{\frac{1}{2}}}g(x)dx=\infty$. In the latter case we apply Lemma \ref{lm_3} to deduce that \eqref{eq0} has no solutions $u\geq 0$. Thus, $g\in L^1_{loc}(B_1)$. Since $Vu\in L^1_{loc}(B_1)$ we deduce $(\K1*u^p)u^q \in L^1_{loc}(B_1)$. \\
Furthermore, since $g\in L^1_{loc}(B_1)$ we can also apply Theorem A to deduce \eqref{eq1d}. This concludes our proof.\qed

\section{A new sub and super-solution method in a nonlocal setting}\label{fss}

In this section we devise a new sub and super-solution method for the semilinear equation with non-local term:
\begin{equation}\label{eqVV}
-\Delta u+V(x) u=(K_{\alpha, \beta}*u^p)u^q\quad\mbox{ in }B_1\setminus\{0\}.
\end{equation}
We point out that our method is different from that recently obtained in \cite{BC25} for the Choquard equation (that is, $q=p-1$ and $\beta=0$) in annular or exterior domains.

\begin{prop}\label{ps}
Assume $N\geq 2$, $p, q>0$, $0\leq V\in \mathcal{C}^{0, \nu}(\overline B_1\setminus\{0\})$ and there exists $\widetilde u, \widetilde U\in \mathcal{C}^2(\overline B_1\setminus\{0\})\cap L^p(B_1)$ such that:
\begin{enumerate}
\item[(i)] $0<\widetilde u\leq \widetilde U$ in $B_1\setminus\{0\}$ and $K_{\alpha, \beta}*\widetilde U^p$ is locally bounded in $B_1\setminus\{0\}$,
\item[(ii)] for all $k\geq 3$ we have
\begin{equation}\label{lap2}
-\Delta \widetilde u(x)+ V(x) \widetilde u(x) \leq 
\left(\,\int\limits_{\frac{1}{k}<|y|<1} K_{\alpha, \beta}(x-y)\widetilde u^p(y) dy\right) \widetilde u^q(x) \quad \mbox{ in }  B_1\setminus B_{\frac{1}{k}},
\end{equation}
\item[(iii)] there holds:
\begin{equation}\label{lap20}
-\Delta \widetilde U(x)+ V(x) \widetilde U(x) \geq 
\left(\,\int\limits_{0<|y|<1} K_{\alpha, \beta}(x-y)\widetilde U^p(y) dy\right) \widetilde U^q(x) \quad \mbox{ in }  B_1\setminus \{0\}. 
\end{equation}
\end{enumerate}
Then, there exists a solution $u\in \mathcal{C}^2(B_1\setminus\{0\})$ of \eqref{eqV} such that $\widetilde u\leq u\leq \widetilde U$ in $B_1\setminus\{0\}$.
\end{prop}

\begin{proof}
We construct a sequence $\{u_k\}_{k\geq 3}$ as follows:\\
We let $u_3:=\widetilde u$ and for all $k\geq 3$ we define $u_{k+1}$ inductively as follows:
\begin{equation}\label{lapk}
\begin{cases}
\displaystyle -\Delta u_{k+1}+ V(x) u_{k+1}=\left(\,\int\limits_{\frac{1}{k}<|y|<1} K_{\alpha, \beta}(x-y)u_k^p(y) dy\right) u_k^q  &\quad \mbox{ in }  B_1\setminus B_{\frac{1}{k}} ,\\[0.2cm]
u_{k+1}=\widetilde u & \quad \mbox{ on }  \partial( B_1\setminus B_{\frac{1}{k}}).
\end{cases}    
\end{equation}
For convenience, we extend $u_{k+1}=\widetilde u$ to $B_{\frac{1}{k}}\setminus\{0\}$.

\noindent{\bf Step 1:} $u_k$ is well defined and $\widetilde u\leq u_{k}\leq u_{k+1}\leq \widetilde U$ in $ B_1\setminus B_{\frac{1}{k}}$ for all $k\geq 3$.\\
Since $u_3=\widetilde u$, it follows by Lemma \ref{rieszlog}(ii) that the mapping
$$
x\longmapsto \left(\,\int\limits_{\frac{1}{3}<|y|<1} K_{\alpha, \beta}(x-y)u_3^p(y) dy\right) u_3^q(x)$$
is locally H\"older continuous in $ B_1\setminus B_{\frac{1}{3}}$. Thus, the existence and uniqueness of a solution $u_4\in \mathcal{C}^2(\overline{ B_1\setminus B_{\frac{1}{3}}})$ to \eqref{lapk} follows from the standard sub and super-solution method (see for instance  \cite[Section 3.2]{P92}).
Also, from \eqref{lap2} and \eqref{lapk} we have 
$$
 -\Delta u_4+ V(x) u_4=\left(\,\int\limits_{\frac{1}{3}<|y|<1} K_{\alpha, \beta}(x-y)u_3^p(y) dy\right) u_3^q \geq  -\Delta u_3+ V(x) u_3 \quad\mbox{ in } B_1\setminus B_{\frac{1}{3}},$$
$$
u_{4}=\widetilde u =u_3 \quad \mbox{ on }  \partial( B_1\setminus B_{\frac{1}{3}}).
$$
By the maximum principle (see, e.g., \cite[Theorem 2, pag 346]{E10}), it follows that $u_4\geq u_3$ in $B_1\setminus B_{\frac{1}{3}}$. 
Further, we have 
$$
\begin{aligned}
 -\Delta u_4+ V(x) u_4 &=\left(\,\int\limits_{\frac{1}{3}<|y|<1} K_{\alpha, \beta}(x-y) u_3^p(y) dy\right) u_3^q \quad\mbox{ in } B_1\setminus B_{\frac{1}{3}},\\[0.1cm]
 &\leq \left(\,\int\limits_{\frac{1}{3}<|y|<1} K_{\alpha, \beta}(x-y)\widetilde U^p(y) dy\right) \widetilde U^q\\[0.1cm]
 &\leq  -\Delta \widetilde U+V(x) \widetilde U\quad\mbox{ in } B_1\setminus B_{\frac{1}{3}},
 \end{aligned}
 $$
and $u_4=\widetilde u\leq \widetilde U$ on $\partial( B_1\setminus B_{\frac{1}{3}})$. By the maximum principle again, we find $u_4\leq \widetilde U$ in $B_1\setminus B_{\frac{1}{3}}$. Hence, we proved $\widetilde u=u_3\leq u_4\leq \widetilde U$ in $B_1\setminus B_{\frac{1}{3}}$. Next, we proceed by induction to deduce the conclusion of this step.

\medskip

\noindent{\bf Step 2:} Let $u(x):=\lim_{k\to \infty}u_k(x)$, $x\in B_1\setminus\{0\}$; then $u\in \mathcal{C}^2(B_1\setminus\{0\})$ is a solution of \eqref{eqVV}.
Since the sequence $\{u_k\}_{k\geq 3}$ is monotone and $\widetilde u\leq u_k\leq \widetilde U$, the above function $u$ is well defined. Let $B$ be an open ball such that $\overline B\subset B_1\setminus\{0\}$. Thus, letting
$$
g_k(x):=\left(\,\int\limits_{\frac{1}{k}<|y|<1} K_{\alpha, \beta}(x-y)u_k^p(y) dy \right) u_k^q(x)\, , \quad x\in  B_1\setminus B_{\frac{1}{k}},
$$
we have by $\widetilde u\leq u_k\leq \widetilde U$ and hypothesis (i) that 
$$
\{u_k\}_{k\geq 3},\, \{g_k\}_{k\geq 3} \mbox{ are bounded in } L^s(B) \mbox{ for all } s > 1.
$$
By standard elliptic regularity we deduce that for all $s>1$ the sequence $\{u_k\}_{k\geq 3}$ is bounded in $W^{2, s}(B)$ and thus, from the compact embedding $W^{2,s}(B)\hookrightarrow \mathcal{C}^{1, \nu}(\overline B)$, $s>N/2$, it is bounded in $ \mathcal{C}^{1, \nu}(\overline B)$ for some $0<\nu<1$. In particular $u\in \mathcal{C}^{1, \nu}_{loc}(B_1\setminus\{0\})$ is a weak solution of \eqref{eqVV}. Next, Lemma \ref{rieszlog}(i) yields 
$$
(K_{\alpha, \beta}*u^p)u^q\in \mathcal{C}^{0, \mu}(\overline B) \quad\mbox{ for some } 0<\mu<1,
$$
and thus, by Schauder estimates, we deduce $u\in \mathcal{C}^2(\overline B)$. Thus, $u\in \mathcal{C}^2(B_1\setminus\{0\})$ is a solution of \eqref{eqVV} which satisfies $\widetilde u\leq u\leq \widetilde U$ in $B_1\setminus\{0\}$.
\end{proof}

\section{Proof of Theorem \ref{thm3} }

(i) We require the following auxiliary result.

\begin{lemma}\label{lemd}
Let $0\leq u\in \mathcal{C}^2(B_1\setminus\{0\})$ and $g\in \mathcal{C}(B_1\setminus\{0\})\cap L^1_{loc}(B_1)$ be such that
\begin{equation}\label{eqdelta1}
-\Delta u=g(x)+m\delta_0\quad\mbox{in }\mathscr{D}'(B_1)\, , m\in \R.
\end{equation}
Then $m\geq 0$ and there exists a harmonic function $h:\overline B_{\frac{1}{2}}\to \R$ such that
\begin{equation}\label{eqdelta2}
u(x)=h(x)+W(x)+mE(|x|)\quad\mbox{in }\overline B_{\frac{1}{2}}\setminus\{0\},
\end{equation}
where 
\begin{equation}\label{eqdelta3}
W(x)=\int_{B_{\frac{1}{2}}} g(x) E(|x-y|)dy.
\end{equation}
Furthermore, we have
\begin{equation}\label{eqdelta4}
\lim_{r\to 0}\frac{\overline u(r)}{E(r)}=m.
\end{equation}
\end{lemma}
\noindent Various versions of the above result are known in the literature (see \cite{BL81}, \cite[Theorem 1.1.]{GT16book} or \cite[Lemma 1]{L96}) especially in the case $g\geq 0$ (so that $u$ becomes super-harmonic in $B_1\setminus\{0\}$). We provide here a complete detail of its proof and emphasize that $g$ may be any sign-changing function that is locally integrable.
\begin{proof}
   From \eqref{eqdelta1}, we have that 
   $$-\Delta u = g \quad\text{in } B_1 \setminus \{0\}.$$
   Taking the spherical average, we obtain:
   $$-(r^{N-1}\overline{u}'(r))' = r^{N-1}\overline{g}(r) \quad \text{for all } 0 < r < 1.$$
   Let $0 < r < \rho < 1$. Integrating the above equality over $[r, \rho]$ yields
   $$r^{N-1}\overline{u}'(r) - \rho^{N-1}\overline{u}'(\rho) = \int_r^{\rho} s^{N-1}\overline{g}(s) \, ds.$$
   Thus, there exists some $c \in \R$ such that
   \begin{equation*}
     r^{N-1}\overline{u}'(r) = c + \frac{1}{\sigma_N}\int_{r < |x| < \rho} g(|x|) \, dx \quad \text{for all } 0 < r < \rho.
   \end{equation*}
   Then, by l'H\^{o}pital's rule, we deduce that there exists
   $$l := \lim\limits_{r \to 0} \frac{\overline{u}(r)}{E(r)} \geq 0.$$
   In particular, we have
   \begin{equation}\label{eqdeltaU}
     \int_{|x| < \varepsilon} u(x) \, dx = 
     \begin{cases}
        \displaystyle o\left(\varepsilon\log\frac{1}{\varepsilon}\right) & \text{if } N = 2, \\[0.3cm]
        \displaystyle o(\varepsilon^2) & \text{if } N \geq 3.
     \end{cases}
    \end{equation}
    On the other hand, by Fubini's theorem, we obtain:
    \begin{align*}
        \int_{B_{\frac{1}{2}}} W(x)\Delta \varphi (x) dx &=  \int_{B_{\frac{1}{2}}} g(y) \left( \int_{B_{\frac{1}{2}}} E(|y - x|)\Delta \varphi(x) dx \right) dy \nonumber \\
        &=  -\int_{B_{\frac{1}{2}}} g(y)\varphi(y) dy \quad\text{for all } \varphi \in C_c^{\infty}(B_{\frac{1}{2}}).
    \end{align*}
    It follows that
    \begin{equation*}
        -\Delta W = g \quad\text{in } \mathscr{D}'(B_{\frac{1}{2}}).
    \end{equation*}
    Thus, 
    $$-\Delta (u - W) = 0 \quad\text{in } \mathscr{D}'(B_{\frac{1}{2}}\setminus \{0\}),$$ 
    meaning that $u - W$ is a harmonic function in $B_{\frac{1}{2}}\setminus \{0\}$. \\
    By Lemma 1.6.1 in \cite{H09}, we have
    \begin{equation}\label{H09}
        \frac{1}{\sigma_N r^{N-1}}\int_{|x| = r} E(|x - y|) d\sigma(x) = 
        \begin{cases}
            E(r) \quad\text{if } |y| < r,\\
            E(y) \quad\text{if } |y| > r.
        \end{cases}
    \end{equation}
    Using \eqref{H09} and the definition of $W$ in \eqref{eqdelta3}, we find
    \begin{align*}
        \overline{W}(r) &= \frac{1}{\sigma_N r^{N-1}}\int_{|x|=r} \left(\int_{|y| < \frac{1}{2}} E(|y - x|)g(y) dy \right) d\sigma(x) \\
        &= \int_{|y| < r} g(y) \left(\frac{1}{\sigma_N r^{N-1}}\int_{|x| = r}E(|y - x|) d\sigma(x) \right) dy \\
        &\quad\quad + \int_{r < |y| < \frac{1}{2}} g(y) \left(\frac{1}{\sigma_N r^{N-1}}\int_{|x| = r}E(|y - x|) d\sigma(x) \right) dy\\
        &= E(r) \int_{|y| < r} g(y) dy + \int_{r < |y| < \frac{1}{2}} g(y)E(y) dy.
    \end{align*}
    Applying the Lebesgue Dominated Convergence Theorem, we deduce
    \begin{equation}\label{eqdeltaN}
        \overline{W}(r) = o(E(r)) \quad\text{as } r \to 0.
    \end{equation}
    Combining \eqref{eqdeltaU} and \eqref{eqdeltaN}, we obtain:
    \begin{equation*}
        \int_{|x| < \varepsilon} |u(x) - W(x)| dx = o(\varepsilon) \quad\text{as } \varepsilon \to 0.
    \end{equation*}
    From Lemma 1.2 in \cite{GT16book}, it follows that $u(x) - W(x) - \gamma E(|x|)$ has a harmonic extension to $B_{\frac{1}{2}}$ for some $\gamma \in \R$. Thus, 
    $$u(x) = h(x) + W(x) + \gamma E(|x|) \quad\text{in } B_{\frac{1}{2}} \setminus \{0\},$$
    where $h: B_{\frac{1}{2}} \to \R$ is harmonic.\\
    Consequently, we find $-\Delta u = g + \gamma\delta_0$ in $\mathscr{D}'(B_{\frac{1}{2}})$ and then from \eqref{eqdelta1}, we conclude that 
    $\gamma = m$. This establishes \eqref{eqdelta2}. Finally, combining \eqref{eqdelta2} and \eqref{eqdeltaN}, we have
    $$\frac{\overline{u}(r)}{E(r)} \to m \quad\text{as } r \to 0.$$
    This concludes the proof.
\end{proof}
\noindent We may now proceed to the proof of part (i) in Theorem \ref{thm3}. Let $0\leq u\in \mathcal{C}^2(B_1\setminus\{0\})$ be a solution to \eqref{eq0} that fulfills \eqref{eqN2}. Then, $u$ satisfies \eqref{eq1d} and thus, by Lemma \ref{lemd}, \eqref{eqdelta1} holds with $g=(K_{\alpha, \beta}*u^p)u^q \in L^1_{loc}(B_{\frac{1}{2}})$. Thus, $u$ satisfies \eqref{eqdelta4} which in light of \eqref{eqN2} yields $m=0$.   Hence,
\begin{equation}\label{eqdelta5}
u(x)=h(x)+\int_{B_{\frac{1}{2}}} g(y)\log\left(\frac{2e}{|x-y|}\right) dy \quad\mbox{ in }\overline B_{\frac{1}{2}}\setminus\{0\},
\end{equation}
where $h:\overline B_{\frac{1}{2}}\to \R$ is a harmonic function. Note that by \eqref{eqN2} and Lemma \ref{lm_2} we have
$$
g(x)\leq c\log^q\frac{2e}{|x|} \quad\mbox{ in }B_1\setminus\{0\}.$$ 
Let us also note that by the above estimates on $g$ and Lemma \ref{lm_2}, the right hand-side of \eqref{eqdelta5} is finite at all points of $B_{\frac{1}{2}}$. Furthermore, by Lemma \ref{rieszlog}(ii) we see that
$$
\overline B_{\frac{1}{2}}\ni x\longmapsto \int_{B_{\frac{1}{2}}} g(y)\log\left(\frac{2e}{|x-y|}\right) dy
\quad \mbox{is H\"older continuous.}
$$
By \eqref{eqdelta5} it follows that $u\in \mathcal{C}^{0, \nu}_{loc}(B_{1})$ for some $0<\nu<1$. In particular,  by the same Lemma \ref{rieszlog}(ii) we deduce $u\in \mathcal{C}^{0, \mu}(\overline B_{\frac{1}{2}})$, for some $0<\mu<1$.\\
Let $w\in \mathcal{C}^2(\overline B_{\frac{1}{2}})$ be the solution of 
$$
\begin{cases}
-\Delta w=g(x)&\quad \mbox{ in }B_{\frac{1}{2}},\\
w=u& \quad \mbox{ on }\partial B_{\frac{1}{2}}.
\end{cases}$$
Since $u$ satisfies \eqref{eqdelta5} with $m=0$, we have  $-\Delta(w-u)=0$ in $\mathscr{D}'(B_{\frac{1}{2}})$, so $w-u$ is harmonic in $B_{\frac{1}{2}}$. From $w=u$ on $\partial B_{\frac{1}{2}}$, it follows that $u=w$ and in particular $u\in \mathcal{C}^2(B_1)$.\medskip\\
(ii) Our main tool is Proposition \ref{ps} in Section \ref{fss}. To this aim, we need to provide a pair of ordered sub and super-solution $(\widetilde u, \widetilde U)$ that satisfy the conditions in Proposition \ref{ps} and 
$$
\lim_{|x|\to 0} \frac{\widetilde u(x)}{\log\frac{2e}{|x|}}=\lim_{|x|\to 0} \frac{\widetilde U(x)}{\log\frac{2e}{|x|}}>0.
$$
Let $p>0$, $q>1$ and set $\widetilde u(x)=m\log\frac{2e}{|x|}$ and $\widetilde U(x)=m\left(\log\frac{2e}{|x|}+\log^\sigma \frac{2e}{|x|}\right)$, where $0<\sigma,m <1$.\\
We first note that $\widetilde U(x)\leq 2m \log\frac{2e}{|x|}$  in $B_1\setminus\{0\}$ and thus, by Lemma \ref{lm_2} we have
$$
(K_{\alpha, \beta}*\widetilde U^p)\widetilde U^q (x)\leq (2m)^{p+q} (K_{\alpha, \beta}* \log^p\frac{2e}{|x|}) \log^q\frac{2e}{|x|}\leq cm^{p+q}\log^q\frac{2e}{|x|} \quad\mbox{ in }B_1\setminus\{0\}.$$
Hence, for small $m>0$ we have
\begin{equation}\label{lap01}
\begin{aligned}
-\Delta \widetilde U(x)+\lambda V(x) \widetilde U(x) & \geq -\Delta \widetilde U(x) \\
&\geq  \frac{m\sigma(1-\sigma)}{|x|^2}\log^{\sigma-2}\frac{2e}{|x|}\\
& \geq cm^{p+q}\log^q\frac{2e}{|x|}\\
& \geq (K_{\alpha, \beta}*\widetilde U^p)\widetilde U^q (x)
\quad \mbox{ in }B_1\setminus\{0\}.   
\end{aligned}
\end{equation}
This shows that $\widetilde U$ satisfies \eqref{lap20}. It remains to check that $\widetilde u$ satisfies condition \eqref{lap2} for $\lambda>0$ small. Indeed, from \eqref{eqN21} with $\varepsilon=q-1>0$ and Lemma \ref{lm_2} we estimate
$$
\begin{aligned}
-\Delta \widetilde u(x)+\lambda V(x) \widetilde u(x) & = \lambda V(x) \widetilde u(x)\\
&\leq cm\lambda \log^{q} \frac{2e}{|x|}\\
&\leq \left(\,\int\limits_{\frac{1}{k}<|y|<1} K_{\alpha, \beta}(x-y)\widetilde u^p(y) dy\right) \widetilde u^q(x) \quad \mbox{ in }  B_1\setminus B_{\frac{1}{k}},
\end{aligned}
$$
for all $k\geq 3$. We now fulfill all the conditions in Proposition \ref{ps} and thus, we obtain a positive solution $u\in \mathcal{C}^2(B_1\setminus\{0\})$ to \eqref{eq0} such that $\widetilde u\leq u\leq \widetilde U$ in $B_1\setminus\{0\}$. In particular, we have 
$$
\lim_{|x|\to 0}\frac{u(x)}{\log\frac{2e}{|x|}}=m>0.
$$
This concludes the proof of Theorem \ref{thm3}.
\qed
\section{Proof of Theorem \ref{thm4}}
Assume $N \geq 3$ and that $0\leq u\in \mathcal{C}^2(B_1\setminus\{0\})$ is a solution of \eqref{eq0} with  $u(x) \simeq |x|^{2-N}$. Then, by the estimates in Lemma \ref{lm_1}(iii) we have:
\begin{equation}\label{thm_4}
    -\Delta u + \lambda V(x) u \simeq \Psi(x) \quad\text{in } B_1 \setminus \{0\},
\end{equation}
where
\begin{equation}\label{thm4e}
\Psi(x) = \left\{
    \begin{aligned}
        &|x|^{-q(N - 2)} \quad&\text{if } p < \frac{N - \alpha}{N - 2}&,\\[0.2cm]
        &|x|^{-q(N - 2)} \quad&\text{if } p = \frac{N - \alpha}{N - 2}&\text{ and } \beta < -1,\\[0.2cm]
        &|x|^{-q(N - 2)}\log\left(\log\frac{2e}{|x|}\right) \quad&\text{if } p = \frac{N - \alpha}{N - 2}&\text{ and } \beta = -1,\\[0.2cm]
        &|x|^{-q(N - 2)}\log^{1 + \beta}\frac{2e}{|x|} \quad&\text{if } p = \frac{N - \alpha}{N - 2}&\text{ and } \beta > -1,\\[0.2cm]
        &|x|^{N -\alpha - (p + q)(N - 2)}\log^{\beta}\frac{2e}{|x|} \quad&\text{if } p > \frac{N - \alpha}{N - 2}&.
    \end{aligned}\right.
\end{equation}
Next, we divide our proof into four cases.\medskip\\ 
{\bf Case 1}: $p < \frac{N - \alpha}{N - 2}$ or $p = \frac{N - \alpha}{N - 2}$ and $\beta < -1$.\\
Since $V(x) = o\left(|x|^{-(q-1)(N-2)}\right)$ as $|x| \to 0$, it follows from \eqref{thm_4} and \eqref{thm4e} that:
\begin{equation}\label{thm4c1}
    -\Delta u \geq C(1 - o(1))|x|^{-q(N - 2)}\quad\text{in } B_1 \setminus \{0\}.
\end{equation}
Then, there exists $0 < \rho < 1$ small such that
\begin{equation*}
    -\Delta u \geq C|x|^{-q(N - 2)} \quad\text{in } B_{\rho}\setminus \{0\}.
\end{equation*}
Thus, by Lemma \ref{lm_3}, we conclude that $|x|^{-q(N-2)} \in L^1_{loc}(B_{\rho})$, which implies that $q < \frac{N}{N - 2}$.\medskip\\ 
{\bf Case 2}: $p = \frac{N - \alpha}{N - 2}$ and $\beta = -1$.\\
In this case, we apply a similar argument to that in \eqref{thm4c1} to conclude the following:
\begin{equation*}
    -\Delta u \geq C|x|^{-q(N - 2)}\log\left(\log\frac{2e}{|x|}\right) \quad\text{in } B_{\rho} \setminus \{0\},
\end{equation*}
where $0 < \rho < 1$ small. Applying Lemma \ref{lm_3}, we deduce that $|x|^{-q(N - 2)}\log\left(\log\frac{2e}{|x|}\right) \in L^1_{loc}(B_{\rho})$, and therefore $q < \frac{N}{N - 2}$.\medskip\\
{\bf Case 3}: $p = \frac{N - \alpha}{N - 2}$ and $\beta > -1$.\\
As above, we apply a similar argument to that in \eqref{thm4c1} to conclude:
\begin{equation*}
    -\Delta u \geq C|x|^{-q(N - 2)}\log^{1 + \beta}\frac{2e}{|x|} \quad\text{in } B_{\rho} \setminus \{0\},
\end{equation*}
where $0 < \rho < 1$ small. Applying Lemma \ref{lm_3} again, we deduce that $|x|^{-q(N - 2)}\log^{1 + \beta}\frac{2e}{|x|} \in L^1_{loc}(B_{\rho})$, and therefore $q < \frac{N}{N - 2}$.\medskip\\
{\bf Case 4}: $p > \frac{N - \alpha}{N - 2}$.\\
We have
\begin{equation*}
    V(x)u(x) = o\left(|x|^{N - \alpha -(p + q)(N - 2)} \log^{\beta}\frac{2e}{|x|}\right) \quad \text{as } |x| \to 0.
\end{equation*}
It follows from \eqref{thm_4} that
\begin{equation}
    -\Delta u \geq C|x|^{N - \alpha - (p + q)(N - 2)}\log^{\beta}\frac{2e}{|x|} \quad\text{in } B_{\rho} \setminus \{0\},
\end{equation}
where $0 < \rho < 1$ small. By Lemma \ref{lm_3}, we derive $|x|^{N - \alpha - (p + q)(N - 2)}\log^{\beta}\frac{2e}{|x|} \in L^1_{loc}(B_{\rho})$ and it follows that
\begin{align*}
\text{either}\quad     p +  q < \frac{2N - \alpha}{N - 2} \quad \text{or} \quad p + q = \frac{2N - \alpha}{N - 2}\text{ and } \beta < -1.
\end{align*}
Since $p > \frac{N - \alpha}{N - 2}$, we also deduce that $q < \frac{N}{N - 2}$. 
\medskip\\
Conversely, assume now that either \eqref{eqN3} or \eqref{eqN4} holds, and let us construct a positive solution $u\in \mathcal{C}^2(B_1\setminus\{0\})$ to \eqref{eq0} such that $u(x)\simeq |x|^{2-N}$. 
To this aim. we need to provide two suitable functions $\widetilde u$ and $\widetilde U$ that fulfill the hypotheses in Proposition \ref{ps}. We split our analysis into three cases.\medskip\\
\noindent{\bf Case A:} $p+q>1$.  If \eqref{eqN3} holds we take $k>0$ such that
\begin{equation}\label{eqkk}
\max\{q(N-2), (p+q)(N-2)-N+\alpha\}-2<k<N-2.   
\end{equation}
If \eqref{eqN4} holds, we choose 
\begin{equation}\label{eqks}
0<\sigma<\min\{\log 2, -(\beta+1)\}.   
\end{equation}
Define next
$\widetilde u(x)=m|x|^{2-N}$ and
$$
\widetilde U(x)=
\begin{cases}
m(|x|^{2-N}+|x|^{-k}) &\quad\mbox{ if \eqref{eqN3} holds},\\
m\left(|x|^{2-N}+|x|^{2-N}\log^{-\sigma} \frac{2e}{|x|}\right)   &\quad\mbox{ if \eqref{eqN4} holds},
\end{cases}$$
where $m>0$ will be precised later.
Then
\begin{equation}\label{lap}
-\Delta \widetilde U(x)=
\begin{cases}
mk(N-2-k)|x|^{-k-2} &\quad\mbox{ if \eqref{eqN3} holds},\\
m|x|^{-N}\log^{-\sigma-1} \frac{2e}{|x|}\left(N-2-\frac{\sigma+1}{\log\frac{2e}{|x|}} \right)  &\quad\mbox{ if \eqref{eqN4} holds}. 
\end{cases}
\end{equation}
Using $\widetilde U(x)\leq  2m|x|^{2-N}$ and Lemma \ref{lm_1}(iii) we find 
$$
(K_{\alpha, \beta}*\widetilde U^p)\widetilde U^q\leq (2m)^{p+q} (K_{\alpha, \beta}* |x|^{-p(N-2)}) |x|^{-q(N-2)}\leq cm^{p+q}\Psi(x),$$
where $\Psi(x)$ is the function defined in \eqref{thm4e}.\\
Hence, from \eqref{eqkk}, \eqref{eqks} and \eqref{lap}, for $0<m<1$ small we deduce
\begin{equation}\label{lap1}
-\Delta \widetilde U(x)+\lambda V(x) \widetilde U(x) \geq -\Delta \widetilde U(x)\geq  (K_{\alpha, \beta}*\widetilde U^p)\widetilde U^q \quad \mbox{ in }B_1\setminus\{0\}.   
\end{equation}
We now fix $0<m<1$ with the above property and note that from Lemma \ref{lm_1}(ii) we have 
$$
\left(\int\limits_{\frac{1}{k}<|y|<1} K_{\alpha, \beta}(x-y)\widetilde u(y)^p dy\right) \widetilde u(x)^q\geq C\Psi(x) \quad \mbox{ for all } x\in B_1\setminus B_{\frac{1}{k}}\, , \,  k\geq 3.
$$
Also, by \eqref{eqV} we have 
$$
-\Delta \widetilde u+\lambda V(x) \widetilde u=\lambda V(x) \widetilde u \leq c\lambda \Psi(x) \quad\mbox{ in } B_1\setminus\{0\}.$$
For $\lambda>0$ small, we can now use Proposition \ref{ps}, and thus there exists $u\in \mathcal{C}^2(B_1\setminus\{0\})$ a solution of \eqref{eq0} such that $\widetilde u\leq u\leq \widetilde U$  in $B_1\setminus\{0\}$. In particular, we deduce $u(x)/|x|^{2-N}\to m$ as $|x|\to 0$.
\medskip\\
\noindent{\bf Case B:} $p+q<1$. Then $q<1$ and by \eqref{eqV} we have $V(x)=O(1)$ as $|x|\to 0$. Thus, $V$ is bounded in $\overline B_1$, so, for some constant $V_0>0$ we have 
\begin{equation}\label{Vz}
V(x) \leq V_0 \quad\mbox{ for all } x\in B_1\setminus\{0\}.
\end{equation}
Let $\mu=\lambda V_0>0$ and $\widetilde u(x)=G_\mu(x)$, where $G_\mu$ is the fundamental solution of $-\Delta+\mu I$. 
 We next fix $0<k<N-2$ that satisfies \eqref{eqkk} and take $\widetilde U(x)=M(|x|^{2-N}+|x|^{-k})$, where $M>1$ is large enough such that $\widetilde U\geq \widetilde u$ in $B_1\setminus\{0\}$. This is always possible in light of Lemma \ref{fs}.
\\
From the definition of $\widetilde u$ we have 
$$
-\Delta \widetilde u+\lambda V(x) \widetilde u\leq -\Delta \widetilde u+\lambda V_0 \widetilde u=0\quad\mbox{ in }B_1\setminus\{0\}, $$
so \eqref{lap2} holds. Also, from $\widetilde U(x)\leq 2M|x|^{2-N}$ in $B_1\setminus\{0\}$ and Lemma \ref{lm_1}(iii) we find
\begin{equation}\label{fsn}
(K_{\alpha, \beta}*\widetilde U^p)\widetilde U^q\leq (2M)^{p+q} (K_{\alpha, \beta}* |x|^{-p(N-2)}) |x|^{-q(N-2)}\leq cM^{p+q}\Psi(x),
\end{equation}
where $\Psi(x)$ is the function defined in \eqref{thm4e}. Thus, for $M>1$ large, from \eqref{eqkk} and \eqref{fsn} we deduce
$$
\begin{aligned}
-\Delta \widetilde U(x)+\lambda V(x) \widetilde U(x) & \geq -\Delta \widetilde U(x)\\
&=Mk(N-2-k)|x|^{-k-2}\\
&\geq CM\Psi(x) \\
&\geq cM^{p+q}\Psi(x) \\
& \geq   (K_{\alpha, \beta}*\widetilde U^p)\widetilde U^q \quad \mbox{ in }B_1\setminus\{0\}.  
\end{aligned}
$$
This allows us to use Proposition \ref{ps} in order to obtain a solution $u\in \mathcal{C}^2(B_1\setminus\{0\})$ of \eqref{eq0} such that $\widetilde u\leq u\leq \widetilde U$. In particular, we have $u(x)\simeq |x|^{2-N}$.
\medskip\\
\noindent{\bf Case C:} $p+q=1$.  Let $0<k<N-2$ satisfy \eqref{eqkk} and $\widetilde U(x)=M(|x|^{2-N}+|x|^{-k})$, where $M>1$ is large.\\
For $0<\rho<1$ small we have 
$$
\begin{aligned}
-\Delta \widetilde U(x)+\lambda V(x) \widetilde U(x) & \geq -\Delta \widetilde U(x)=Mk(N-2-k)|x|^{-k-2}\\
&> CM\Psi(x)  \geq   (K_{\alpha, \beta}*\widetilde U^p)\widetilde U^q \quad \mbox{ in }B_\rho\setminus\{0\}.  
\end{aligned}
$$
By the continuity of $V$ in $\overline B_1\setminus B_\rho$, we may  take $\lambda>0$ large so that 
$$
\begin{aligned}
-\Delta \widetilde U(x)+\lambda V(x) \widetilde U(x) & \geq \lambda M V(x)|x|^{-k}\geq  
 CM\Psi(x)  \geq   (K_{\alpha, \beta}*\widetilde U^p)\widetilde U^q \quad \mbox{ in }B_1\setminus B_\rho.  
\end{aligned}
$$
Hence, for $\lambda>0$ large we have
\begin{equation}\label{eqM}
-\Delta \widetilde U(x)+\lambda V(x) \widetilde U(x) \geq   (K_{\alpha, \beta}*\widetilde U^p)\widetilde U^q \quad \mbox{ in }B_1\setminus\{0\}.  
\end{equation}
For such a value of $\lambda>0$ we let $\mu=\lambda V_0$, where $V_0>0$ is a constant that fulfills \eqref{Vz}, and take $\widetilde u=G_\mu$. Then $-\Delta \widetilde u+\lambda V(x)\widetilde u\leq 0$ in $B_1\setminus\{0\}$. Note that since $p+q=1$, the size of $M>1$ does not affect the validity of \eqref{eqM}. We may now take $M>1$ sufficiently large to ensure $\widetilde U\geq \widetilde u$ in $B_1\setminus\{0\}$. Hence, we fulfill all the hypotheses in Proposition \ref{ps} and thus deduce the existence of a positive solution $u\in \mathcal{C}^2(B_1\setminus\{0\})$ to \eqref{eq0} with the property $u(x)\simeq |x|^{2-N}$.
This concludes the proof of Theorem \ref{thm4}.
\qed
\medskip\\
\noindent{\bf Proof of Corollary \ref{cor2}.} Let $\lambda>0$ be fixed and assume first that $p+q\neq 1$. Let $0<k<N-2$ and $\sigma>0$ that satisfy \eqref{eqkk} and \eqref{eqks} respectively.
Define 
$\widetilde u(x)=M G_\lambda(x)$ and
$$
\widetilde U(x)=
\begin{cases}
M(G_\lambda(x)+|x|^{-k}) &\quad\mbox{ if \eqref{eqN3} holds},\\
M(G_\lambda(x)+|x|^{2-N}\log^{-\sigma} \frac{2e}{|x|})   &\quad\mbox{ if \eqref{eqN4} holds}. 
\end{cases}$$
It is easy to see that $-\Delta \widetilde u+\lambda \widetilde u=0$ in $B_1\setminus\{0\}$ while 
$$
-\Delta \widetilde U(x)+\lambda \widetilde U(x) \geq -\Delta \widetilde U(x)\geq  (K_{\alpha, \beta}*\widetilde U^p)\widetilde U^q \quad \mbox{ in }B_1\setminus\{0\},$$
provided $0<M<1$ is small if $p+q>1$ and $M>1$ is large if $p+q<1$. Hence, we can invoke Proposition \ref{ps} in order to obtain a positive solution $u\in \mathcal{C}^2(B_1\setminus\{0\})$ of \eqref{lala} such that $\widetilde u\leq u\leq \widetilde U$. In particular, by Lemma \ref{fs} we have 
$$
\lim_{|x|\to 0} \frac{u(x)}{|x|^{2-N}}=\lim_{|x|\to 0} \frac{M G_\lambda(x)}{|x|^{2-N}}=\frac{M}{(N-2)\sigma_N}>0.
$$
If $p+q=1$, we can define $\widetilde u$ and $\widetilde U$ as above in which we take $M=1$. 
As in Case C of the proof of Theorem \ref{thm4} we need $\lambda>0$ sufficiently large in order for \eqref{eqM} (with $V\equiv 1$) to hold. The same approach as above yields the conclusion. \qed

\section*{Declarations}

\subsection*{Conflict of interest}

On behalf of all authors, the corresponding author states that there is no conflict of interest.

\subsection*{Data Availability}

Data sharing not applicable to this article as no datasets were generated or analysed during the current study.

\end{document}